\documentclass[12pt]{amsart}

\usepackage{amssymb,amsmath,latexsym,enumerate,graphicx,bbm,mathptmx,microtype,
cite,mathtools,amsthm}
\usepackage{hyperref,url}
\usepackage{multicol,enumitem}
\usepackage{verbatim,color}
\usepackage[a4paper,left=2.5cm, right=2.5cm, top=2.5cm, bottom=2.5cm]{geometry}

\makeatletter
\@namedef{subjclassname@2010}{%
  \textup{2010} Mathematics Subject Classification}
\makeatother

\newtheorem{prop}{Proposition}[section]
\newtheorem{thm}[prop]{Theorem}
\newtheorem{lem}[prop]{Lemma}
\newtheorem{cor}[prop]{Corollary}

\theoremstyle{definition}
\newtheorem{defin}[prop]{Definition}
\theoremstyle{remark}

\newtheorem{exa}[prop]{Example}

\newcommand\bI{\mathbf{I}}
\newcommand\Tstrut{\rule{0pt}{2.6ex}}
\newcommand\Bstrut{\rule[-0.9ex]{0pt}{0pt}}

\renewcommand\emptyset{\varnothing}

\newcommand\Em{\operatorname{Em}}
\newcommand\oc{\operatorname{oc}}
\newcommand\ul{\underline}
\newcommand{\Inv}{\mathbf{I}} 

\newcommand\W{W}

\newcommand\proj{\operatorname{proj}}

\newcommand\Des{\operatorname{Des}}
\newcommand\Asc{\operatorname{Asc}}

\begin{document}

\baselineskip=17pt

\title{Wilf equivalences between vincular patterns in inversion sequences}

\author{Juan S. Auli}
\address{Department of Mathematics\\
         Dartmouth College\\
         Hanover, NH 03755\\
         U.S.A.}
\email{juan.s.auli.gr@dartmouth.edu}

 \author{Sergi Elizalde}
 \address{Department of Mathematics\\
          Dartmouth College\\
          Hanover, NH 03755\\
          U.S.A.}
 \email{sergi.elizalde@dartmouth.edu}
 \urladdr{http://math.dartmouth.edu/~sergi}

\begin{abstract}

Inversion sequences are finite sequences of non-negative integers, where the value of each entry is bounded from above by its position.
Patterns in inversion sequences have been studied by Corteel--Martinez--Savage--Weselcouch and Mansour--Shattuck in the classical case, where patterns can occur in any positions, and by Auli--Elizalde in the consecutive case, where only adjacent entries can form an occurrence of a pattern. These papers classify classical and consecutive patterns of length 3 into Wilf equivalence classes according to the number of inversion sequences avoiding them.

In this paper we consider vincular patterns in inversion sequences, which, in analogy to Babson--Steingr\'{\i}msson patterns in permutations, require only certain entries of an occurrence to be adjacent, and thus generalize both classical and consecutive patterns. Solving a conjecture of Lin and Yan, we provide a complete classification of vincular patterns of length 3 in inversion sequences into Wilf equivalence classes, and into more restrictive classes that consider the number of occurrences of the pattern and the positions of such occurrences. We find the first known instance of patterns in inversion sequences where these two more restrictive classes do not coincide.

\end{abstract}

\keywords{Inversion sequence, pattern avoidance, vincular pattern, Wilf equivalence.}

\subjclass[2010]{Primary 05A05; Secondary 05A19.}


\maketitle

\section{Introduction}\label{sec:intro}

Let $S_n$ denote the set of permutations of $[n]=\{1,2,\dots,n\}$.
A permutation $\pi\in S_n$ can be encoded by the sequence $e_1e_2\dots e_n$, where $e_{i}=\left|\{j:j<i\textnormal{ and }\pi_{j}>\pi_{i}\}\right|$ is the number of inversions between the $i$th entry of $\pi$ and entries to its left.
This encoding provides a bijection between $S_n$ and the set of {\em inversion sequences}
$$\bI_n=\{e_1e_2\dots e_n: 0\leq e_{i}<i \text{ for all }i\}.$$

This bijection prompted Corteel, Martinez, Savage, and Weselcouch~\cite{MartinezSavageI}, as well as
Mansour and Shattuck~\cite{MansourShattuck}, to initiate the study of patterns in inversion sequences, with the goal of informing the study of patterns in permutations. Their enumeration of inversion sequences avoiding classical patterns of length 3 yielded interesting connections to well-known sequences, including Bell numbers, Fibonacci numbers, and Schr{\"o}der numbers. In addition, they classified classical patterns of length $3$ in inversion sequences according to the number of permutations of each length that avoid them.

The work in~\cite{MartinezSavageI,MansourShattuck}, together with the developing interest in consecutive patterns in permutations~\cite{Elizalde, ElizaldeNoy}, motivated the authors to begin an analogous study of consecutive patterns in inversion sequences~\cite{AuliElizalde}. Results in~\cite{AuliElizalde} include the enumeration of inversion sequences avoiding consecutive patterns of length 3, as well as the classification of consecutive patterns of length $3$ and $4$ into equivalence classes according to the number of inversion sequences avoiding them, and more generally, the number of those containing them a specific number of times or in specific positions.

In this paper we consider vincular patterns in inversion sequences, which provide a common generalization of classical and consecutive patterns studied in~\cite{MartinezSavageI,MansourShattuck} and~\cite{AuliElizalde}, respectively.
To introduce the notion of a vincular pattern, first define the {\em reduction} of a word $w=w_1w_2\dots w_k$ over the integers to be the word obtained by replacing all instances of the $i$th smallest entry of $w$ with $i-1$, for all $i$. For example, the reduction of $3253$ is $1021$.

\begin{defin} A {\em vincular pattern} is a sequence $p=p_{1}p_{2}\dots p_{r}$ where some disjoint subsequences of two or more adjacent entries may be underlined, satisfying $p_{i}\in\left\{0,1,\ldots,r-1\right\}$ for each $i$, where any value $j>0$ can only appear in $p$ only if $j-1$ appears as well.

An inversion sequence $e$ {\em contains} the vincular pattern
$p$ if there is a subsequence $e_{i_{1}}e_{i_{2}}\dots e_{i_{r}}$ of $e$ whose reduction is $p$, and such that $i_{s+1}=i_{s}+1$ whenever $p_{i_{s}}$ and $p_{i_{s+1}}$ are part of the same underlined subsequence.
In such case, the subsequence $e_{i_{1}}e_{i_{2}}\dots e_{i_{r}}$ is called an {\em occurrence} of $p$ in positions $\{i_{1},i_{2},\ldots,i_{r}\}$.
Denote by $\oc(p,e)$ the number of occurrences of $p$ in $e$, and let
$$\bI_n(p,m)=\{e\in\bI_n:\oc(p,e)=m\}.$$

If $\oc(p,e)=0$, then we say that $e$ {\em avoids} $p$.
We use the simpler notation $\bI_n(p)$ for the set $\bI_n(p,0)$ of inversion sequences that avoid $p$.
\end{defin}

In an occurrence of a vincular pattern, underlined subsequences are required to be in adjacent positions.
A vincular pattern $p=p_{1}p_{2}\dots p_{r}$ where no entries are underlined is a {\em classical pattern}; whereas a vincular pattern of the form $p=\ul{p_{1}p_{2}\dots p_{r}}$ is a {\em consecutive pattern}.
In analogy to vincular permutation patterns, introduced by Babson and Steingr\'{\i}msson~\cite{Babson,steingrimsson2010} (who called them {\em generalized patterns}), vincular patterns in inversion sequences generalize both classical and consecutive patterns.

\begin{exa}\label{exa:vinc}
The inversion sequence $e=0013204\in\bI_{7}$ avoids the classical pattern $201$, the consecutive pattern $\underline{000}$, and
the vincular pattern $\underline{01}1$, but it contains the classical pattern $010$, the consecutive pattern $\underline{021}$, and the vincular pattern $\underline{00}0$. For example, $e_{2}e_{3}e_{6}$ is an occurrence of $010$, $e_{3}e_{4}e_{5}$ is an occurrence of $\underline{021}$, and $e_{1}e_{2}e_{6}$ is an occurrence of $\underline{00}0$. One can check that $\oc(012,e)=12$, $\oc(\ul{01}2,e)=4$, $\oc(0\ul{12},e)=2$, and $\oc(\ul{012},e)=1$.
\end{exa}

Unlike patterns in permutations (see~\cite[Ch.~4]{BonaBook} or~\cite[Ch.~1]{KitaevBook} for the basic definitions), patterns in inversion sequences may have repeated entries. Henceforth, the word {\em patterns} will refer to vincular patterns in inversion sequences, unless otherwise stated.

It will be convenient to draw inversion sequences $e=e_1e_2\dots e_n$ as underdiagonal lattice paths on the plane, from the origin to the line $x=n$, consisting of unit vertical steps $(0,1)$ and $(0,-1)$, and unit horizontal steps $(1,0)$. Each entry $e_{i}$ is represented by a horizontal step between the points $(i-1,e_{i})$ and $(i,e_{i})$. The necessary vertical steps are then inserted to make the path connected, see Figure~\ref{fig:path_represent} for an example.

\begin{figure}[htb]
	\begin{center}
	\includegraphics[scale=0.6]{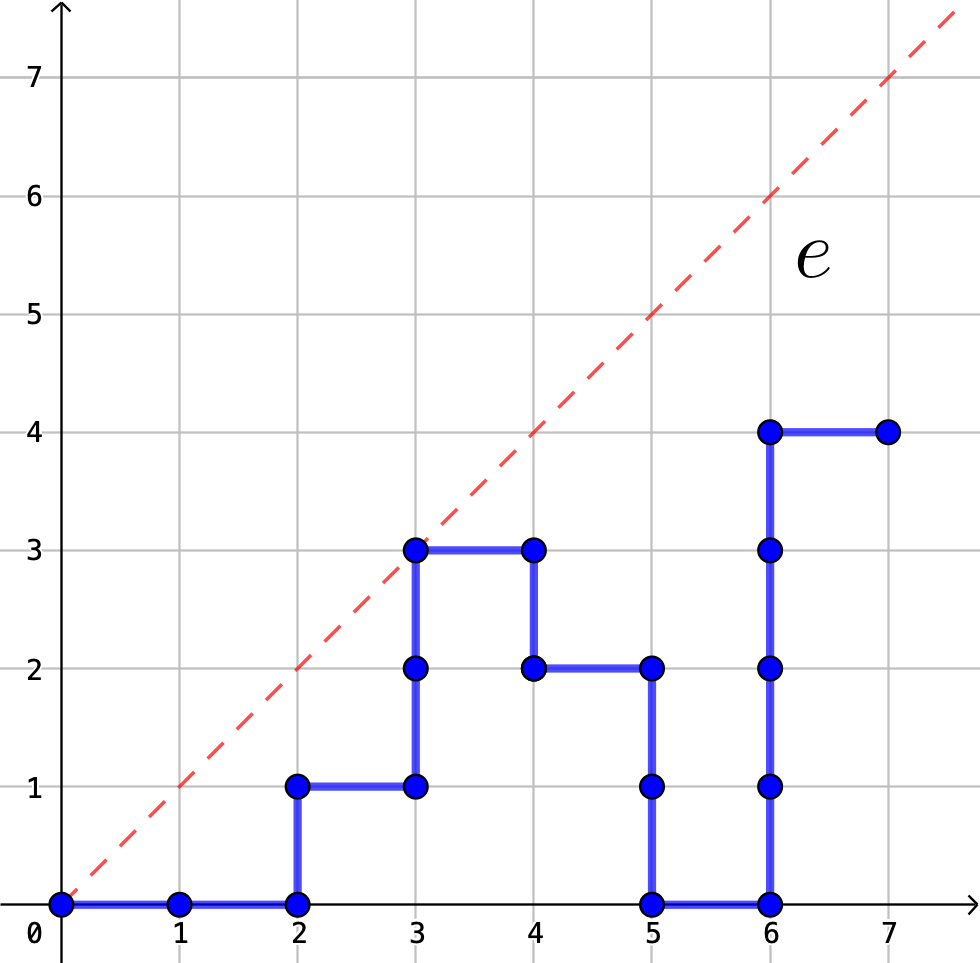}
		\par\end{center}

	\protect\caption{Visualization of $e=0013204\in\bI_{7}$ as a lattice path.\label{fig:path_represent}}
\end{figure}

Next we extend the notion of Wilf equivalence from~\cite{AuliElizalde,MartinezSavageI,MansourShattuck} to vincular patterns. Two patterns are in the same Wilf equivalence class if they are avoided by the same number of inversion sequences of each length. We also introduce more restrictive equivalence relations that also consider the number of occurrences of the patterns and the positions of such occurrences.

\begin{defin}\label{def:Wilf_vinc} Let $p$ and $p'$ be vincular patterns. We say that $p$ and $p'$ are
\begin{itemize}
\item
{\em Wilf equivalent\/}, denoted by $p\sim p'$, if
$\left|\bI_{n}\left(p\right)\right|=\left|\bI_{n}\left(p'\right)\right|$,
for all $n$;
\item
{\em strongly Wilf equivalent\/}, denoted by $p\stackrel{s}{\sim}p'$,
if $\left|\bI_{n}\left(p,m\right)\right|=\left|\bI_{n}\left(p',m\right)\right|$, for all $n$ and $m$.
\end{itemize}
\end{defin}

Denote by $\binom{[n]}{r}$ the set of $r$-element subsets of $[n]$. We use the subscript $<$ on a set to indicate that the elements of a set are listed in increasing order.

Given $n\geq 0$, a vincular pattern $p$ of length $r$, and a set $S\subseteq\binom{[n]}{r}$, we define
\begin{equation}\label{eq:def_super_strongly}
\bI_{n}(p,S) = \left\{e\in\bI_{n}: e_{i_{1}}e_{i_{2}}\ldots e_{i_{r}} \textnormal{ is an occurrence of } p \textnormal{ if and only if } \{i_{1},i_{2},\ldots, i_{r}\}_<\in S\right\}.
\end{equation}
In other words, $\bI_{n}(p,S)$ is the set of inversion sequences of length $n$ whose occurrences of $p$ are indexed by elements of $S$. In particular, $\bI_{n}(p,\emptyset)=\bI_{n}(p)$.

\begin{exa}
There are exactly 6 inversion sequences of length 6 whose occurrences of $1\underline{01}$ are in positions $S=\{\{2,4,5\},\{3,4,5\}\}$. Namely,
\[
\bI_{6}\left(1\underline{01},S\right) = \left\{{011010}, {011011}, {011012}, {011013}, {011014}, {011015}\right\}.
\]
\end{exa}

\begin{defin}\label{def:ssWilf_vinc}
Let $p$ and $p'$ be vincular patterns of length $r$. We say that $p$ and $p'$ are {\em super-strongly Wilf equivalent}, denoted  by
$p\stackrel{ss}{\sim}p'$, if $\left|\bI_{n}(p,S)\right|=\left|\bI_{n}(p',S)\right|$, for all $n$ and all $S\subseteq\binom{[n]}{r}$.
\end{defin}

We use the term {\em generalized Wilf equivalence} to refer to an equivalence of any one of the three types from Definitions~\ref{def:Wilf_vinc} and~\ref{def:ssWilf_vinc}. These three notions of equivalence between vincular patterns extend those defined by the authors for consecutive patterns~\cite{AuliElizalde}. As suggested by their names, $p\stackrel{ss}{\sim}p'$ implies $p\stackrel{s}{\sim}p'$, which in turn implies $p\sim p'$.

\section{Summary of Results}\label{sec:summary_results_vincular}

The main goal of this paper is to describe all generalized Wilf equivalences between vincular patterns of length 3. Equivalences between classical patterns were described in~\cite{MartinezSavageI}, whereas equivalences between consecutive patterns appear in~\cite{AuliElizalde}. The next theorem gives a complete list of generalized Wilf equivalences between vincular patterns that are neither classical nor consecutive.
Such patterns will be called {\em hybrid vincular patterns}.

\begin{thm}\label{EquivVinc}  A complete list of generalized Wilf equivalences between hybrid vincular patterns of length 3 is as follows: \vspace{-6pt}
\begin{multicols}{2}
 \begin{enumerate}[label=(\roman*),itemsep=1ex,leftmargin=2cm]
 \item $\protect\underline{01}0 \sim\protect\underline{01}1$.

 \item $\protect\underline{10}0 \sim\protect\underline{10}1$.

 \item $1\protect\underline{01}\stackrel{s}{\sim} 1\protect\underline{10}$.

 \item $2\underline{01}\sim 2\underline{10}$.
 \end{enumerate}
\end{multicols}
\end{thm}

An independent proof of Theorem~\ref{EquivVinc}(i) has recently been given by Lin and Yan~\cite{Scooped}\footnote{Lin and Yan's work~\cite{Scooped} appeared online while this paper was being written up.}. In the same paper, they also conjecture the Wilf equivalence $2\underline{01}\sim 2\underline{10}$, corresponding to our Theorem~\ref{EquivVinc}(iv).

There are 26 hybrid vincular patterns of length 3, which fall into 22 Wilf equivalence classes and 25 strong Wilf equivalence classes. This is in  contrast with the case of hybrid vincular permutation patterns of length 3, where the 12 patterns fall into 2 Wilf equivalence classes, as shown by Claesson~\cite{Claesson}, and 5 strong Wilf equivalence classes (with all equivalences arising from trivial symmetries).

Corteel et al.\ \cite{MartinezSavageI} prove that the only Wilf equivalences between classical patterns of length 3 are $201\sim 210$ and $101\sim 110$, whereas the authors~\cite{AuliElizalde} show that the only Wilf equivalent consecutive patterns of length 3 are $\underline{100}\stackrel{ss}{\sim}\underline{110}$.

In addition to the above equivalences, there are also some Wilf equivalences between hybrid vincular patterns and classical patterns that follow from known results. Corteel et al.\ \cite{MartinezSavageI} proved that $\left|\bI_{n}\left(001\right)\right|=2^{n-1}$, $\left|\bI_{n}\left(011\right)\right|=B_n$ (the $n$th Bell number), and $\left|\bI_{n}\left(101\right)\right|=\left|\bI_{n}\left(110\right)\right|=\left|S_{n}\left(1\underline{23}4\right)\right|$, which denotes the number of permutations avoiding the vincular permutation pattern $1\underline{23}4$.
On the other hand, Lin and Yan~\cite{Scooped} show that $\left|\bI_{n}\left(0\underline{01}\right)\right|=2^{n-1}$, $\left|\bI_{n}\left(0\underline{12}\right)\right|=B_n$, and $\left|\bI_{n}\left(\underline{12}0\right)\right|=\left|S_{n}\left(1\underline{23}4\right)\right|$. Therefore, we have the Wilf equivalences
$$001\sim 0\underline{01},\qquad 011\sim 0\underline{12},\qquad 101\sim 110\sim \underline{12}0.$$
The first equivalence generalizes, in fact, to the equality $\bI_{n}\left(001\right)=\bI_{n}\left(0\underline{01}\right)$.

Brute force computations for small values of~$n$ show that there are no more generalized Wilf equivalences between vincular patterns other that the ones mentioned above. Thus, Theorem~\ref{EquivVinc} completes the classification of all vincular patterns of length 3 into generalized Wilf equivalence classes of each type. We summarize all these equivalences in Table~\ref{tab:vinc_class}. In total, there are 52 vincular patterns of length 3: 13 consecutive, 13 classical, and 26 hybrid. These patterns fall into 42 Wilf equivalence classes, 50 strong Wilf equivalent classes and 51 super-strong Wilf equivalence classes.

\begin{table}[h]
\footnotesize
\begin{centering}
\resizebox{\columnwidth}{!}{%
 \begin{tabular}{ c@{\hskip 0.3cm}l@{\hskip 0.3cm}c@{\hskip 0.3cm}l}
 \hline
 Pattern $p$ & $\left|\bI_{n}(p)\right|$ counted by & OEIS~\cite{OEIS} & $\left|\bI_{n}(p)\right|$ for $1\le n\le 10$
 \Tstrut\Bstrut\\
 \hline
  $001\sim 0\underline{01}$ & $2^{n-1}$ & A000079 & $1,2,4,8,16,32,64,128,256,512$ \Tstrut\\

  $011\sim 0\underline{12}$ & Bell numbers & A000110 & $1,2,5,15,52,203,877,4140,21147,115975$ \Tstrut\\

 $\underline{01}0\sim \underline{01}1$ & Fishburn numbers & A022493 & $1,2,5,15,53,217,1014,5335,31240,201608$ \Tstrut\\

 $101\sim 110\sim \underline{12}0$ & $\left|S_{n}\left(1\underline{23}4\right)\right|$ & A113227 & $1,2,6,23,105,549,3207,20577,143239,1071704$ \Tstrut\\

 $\underline{10}0\sim \underline{10}1$ & (see Proposition~\ref{prop:GF}) & New & $1,2,6,23,106,567,3440,23286,173704,1414102$ \Tstrut\\

 $1\underline{01}\stackrel{s}{\sim} 1\underline{10}$ & ? & New & $1,2,6,23,107,584,3655,25790,202495,1750763$ \Tstrut\\

$201\sim 210$ & reccurrence~\cite[Eq.\ (5)]{MartinezSavageI} & A263777 & $1,2,6,24,118,674,4306,29990,223668,1763468$ \Tstrut\\

 $2\underline{01}\sim 2\underline{10}$ & ? & New & $1,2,6,24,118,680,4460,32634,262536,2296532$ \Tstrut\\

 $\underline{100}\stackrel{ss}\sim \underline{110}$ & reccurrence~\cite[Prop.\ 3.4]{AuliElizalde} & A328441 & $1,2,6,23,109,618,4098,31173,267809,2565520$ \Tstrut\Bstrut\\
 \hline
 \end{tabular}
 }
 \end{centering}
\caption[Complete list of generalized Wilf equivalences between vincular patterns of length 3]{Complete list of generalized Wilf equivalences between vincular patterns of length~3. The patterns are listed from least avoided to most avoided in inversion sequences of length~$10$.}\label{tab:vinc_class}
\end{table}

A consequence of Theorem~\ref{EquivVinc} and Table~\ref{tab:vinc_class} is that $1\underline{01}$ and $1\underline{10}$ are the only nonconsecutive vincular patterns of length 3 that are strongly Wilf equivalent. The existence of such a pair is somewhat surprising, given the exacting requirement for  nonconsecutive vincular patterns to be strongly Wilf equivalent.

Even more striking is the fact that $1\underline{01}$ and $1\underline{10}$ are strongly Wilf equivalent but not super-strongly Wilf equivalent, making  these patterns the only known instance of vincular patterns in inversion sequences with this property.
Compare this to the fact, shown in~\cite{AuliElizalde}, that for consecutive patterns of length up to 4, strong Wilf equivalence and super-strong Wilf equivalence classes coincide. In fact, it is shown in~\cite{AuliElizaldeII} that this coincidence extends to so-called consecutive patterns of relations of length~3.

Our second result provides a generalization of Theorem~\ref{EquivVinc}(iv) to patterns of arbitrary length.

\begin{thm}\label{thm:gen_2-01_2-10} Let $p=\underline{p_{1}p_{2}\ldots p_{r}}$ be a consecutive pattern and let $d=\max_{i}\left\{p_{i}\right\}$. Then the vincular patterns $(d{+}1)\underline{p_{1}p_{2}\ldots p_{r}}$ and $(d{+}1)\underline{p_{r}p_{r-1}\ldots p_{1}}$ are Wilf equivalent.
\end{thm}

Upcoming work of the first author~\cite{AuliThesis} provides the enumeration of $\left|\bI_{n}\left(p\right)\right|$ for some hybrid vincular patterns $p$ of length~3, proving that, in some cases, the sequence $\left|\bI_{n}\left(p\right)\right|$ also counts other well-known combinatorial structures.
These results,  summarized in Table~\ref{tab:vinc_pattern}, have been obtained independently by Lin and Yan~\cite{Scooped}.

\begin{table}[htb]
\footnotesize
\begin{centering}
\resizebox{\columnwidth}{!}{%
 \begin{tabular}{c@{\hskip 0.3cm}l@{\hskip 0.3cm}c@{\hskip 0.3cm}l }
 \hline
 Pattern $p$ & $\left|\bI_{n}(p)\right|$ counted by & OEIS~\cite{OEIS} & $\left|\bI_{n}\left(p\right)\right|$ for $1\le n\le 10$
 \Tstrut\Bstrut\\
 \hline
   $0\underline{01}$ & $2^{n-1}$  & A000079 & $1,2,4,8,16,32,64,128,256,512$ \Tstrut\\
 $0\underline{12}$ & Bell numbers & A000110 & $1,2,5,15,52,203,877,4140,21147,115975$ \Tstrut\\
$\underline{01}0 \sim\underline{01}1$  & Fishburn numbers & A022493& $1,2,5,15,53,217,1014,5335,31240,201608$\Tstrut\\
 $0\underline{21}$  & certain Bell-like numbers & A091768
 & $1,2,6,22,92,426,2150,11708,68282,423948$ \Tstrut\\
   $\underline{02}1$  & $\left|S_{n}\left(2\underline{41}3\right)\right|$ (semi-Baxter) & A117106& $1,2,6,23,104,530,2958,17734,112657,750726$ \Tstrut\\
  $\underline{12}0$  & $\left|S_{n}\left(1\underline{23}4\right)\right|$ & A113227& $1,2,6,23,105,549,3207,20577,143239,1071704$ \Tstrut\\
  $0\underline{11}$  & recurrence $a_{n+1} = n\,a_{n} + a_{n-1}$
 & A102038 & $1,2,5,17,73,382,2365,16937,137861,1257686$ \Tstrut\Bstrut\\
 \hline
 \end{tabular}
 }
 \end{centering}
 \caption[Hybrid vincular patterns $p$ for which $\left|\bI_{n}\left(p\right)\right|$ appears in the OEIS~\cite{OEIS}]{Hybrid vincular patterns $p$ of length 3 for which $\left|\bI_{n}\left(p\right)\right|$ appears in the OEIS~\cite{OEIS}. The patterns are listed from least avoided to most avoided in inversion sequences of length~$10$.}\label{tab:vinc_pattern}
\end{table}

The rest of the paper is devoted to the proofs of Theorems~\ref{EquivVinc} and~\ref{thm:gen_2-01_2-10}. Our methods include bijections, sieve methods, and the use of generating trees. In the process, we also derive a functional equation satisfied by the generating function of the sequence $\left|\bI_{n}\left(\ul{10}0\right)\right|$.

\section{Proofs of Wilf Equivalences}\label{sec:wilf_equivalences}

\subsection{The patterns $2\protect\underline{01}$ and $2\protect\underline{10}$}\label{subsec:2_01_and_2_10}
In this subsection we prove Theorem~\ref{EquivVinc}(iv), which was conjectured by Lin and Yan~\cite{Scooped}. Our proof is bijective, and it uses the following notion. We say that a position $j$ is a {\em weak left-to-right maximum} of an inversion sequence $e$ if $e_{i}\leq e_{j}$ for all $i<j$. We denote the set of weak left-to-right maxima of $e$ by $\W(e)$. Note that $1\in \W(e)$ for every nonempty inversion sequence $e$.

\begin{exa} The set of weak left-to-right maxima of $e=001210031012$ is $\W(e)=\left\{1,2,3,4,8\right\}$, see Figure~\ref{fig:2-01_2-10}(left).
\end{exa}

\begin{prop}\label{prop:2_01_and_2_10} The patterns $2\underline{01}$ and $2\underline{10}$ are Wilf equivalent.
\end{prop}

\begin{proof}
Define a map $\varphi:\bI_{n}\rightarrow\bI_{n}$ as follows. Given $e\in\bI_{n}$ with $\W(e)=\left\{w_{1},w_{2},\ldots,w_{t}\right\}_<$,  define $\varphi(e)=e'$ as
\begin{equation}
\label{eq:eprime}
e'=(e_{w_{1}-1}e_{w_{1}-2}\dots e_{1})e_{w_{1}}(e_{w_{2}-1}e_{w_{2}-2}\dots e_{w_{1}+1})e_{w_{2}}\dots e_{w_{t}}(e_{n}e_{n-1}\dots e_{w_{t}+1}).
\end{equation}
In other words, $\varphi$ reverses the blocks between the elements of $\W(e)$, see Figure~\ref{fig:2-01_2-10} for an example.

\begin{figure}[htb]
\begin{center}
\includegraphics[scale=0.55]{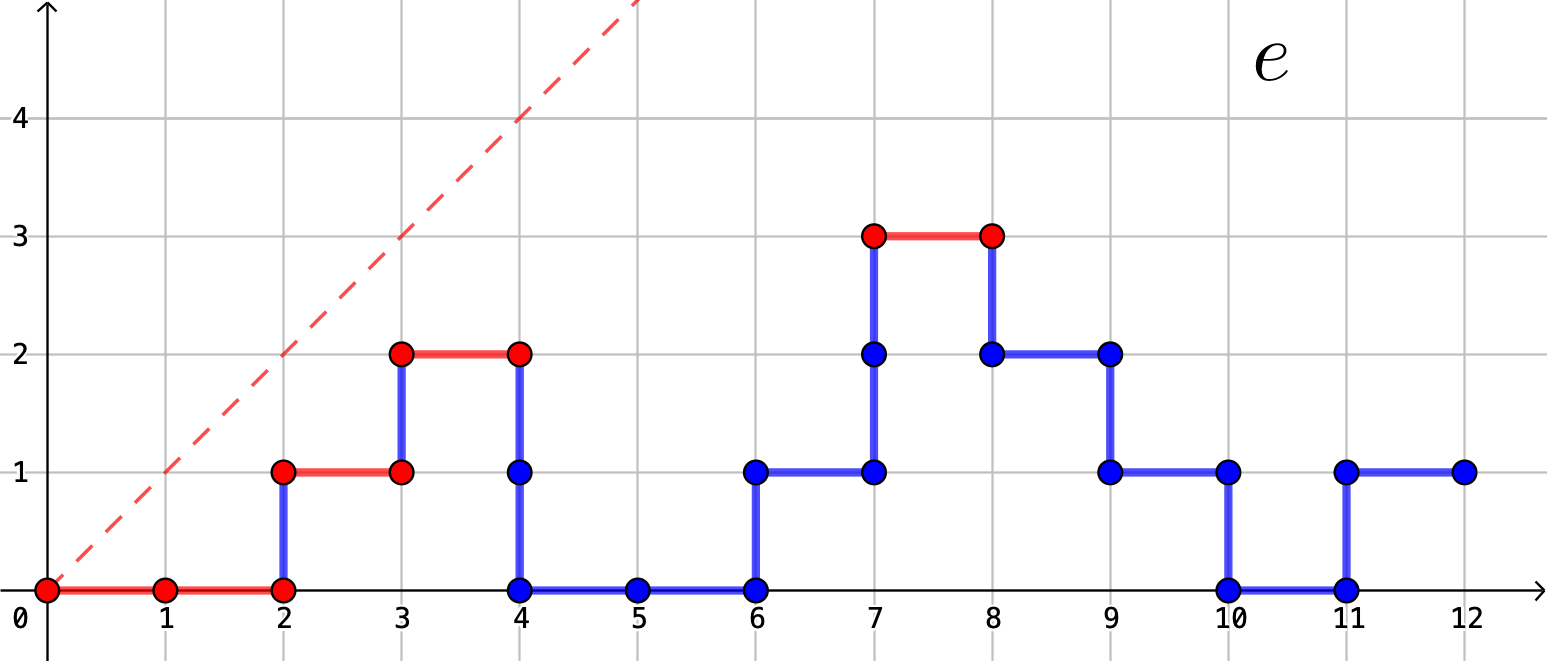}
\hspace*{0.2cm}
\includegraphics[scale=0.55]{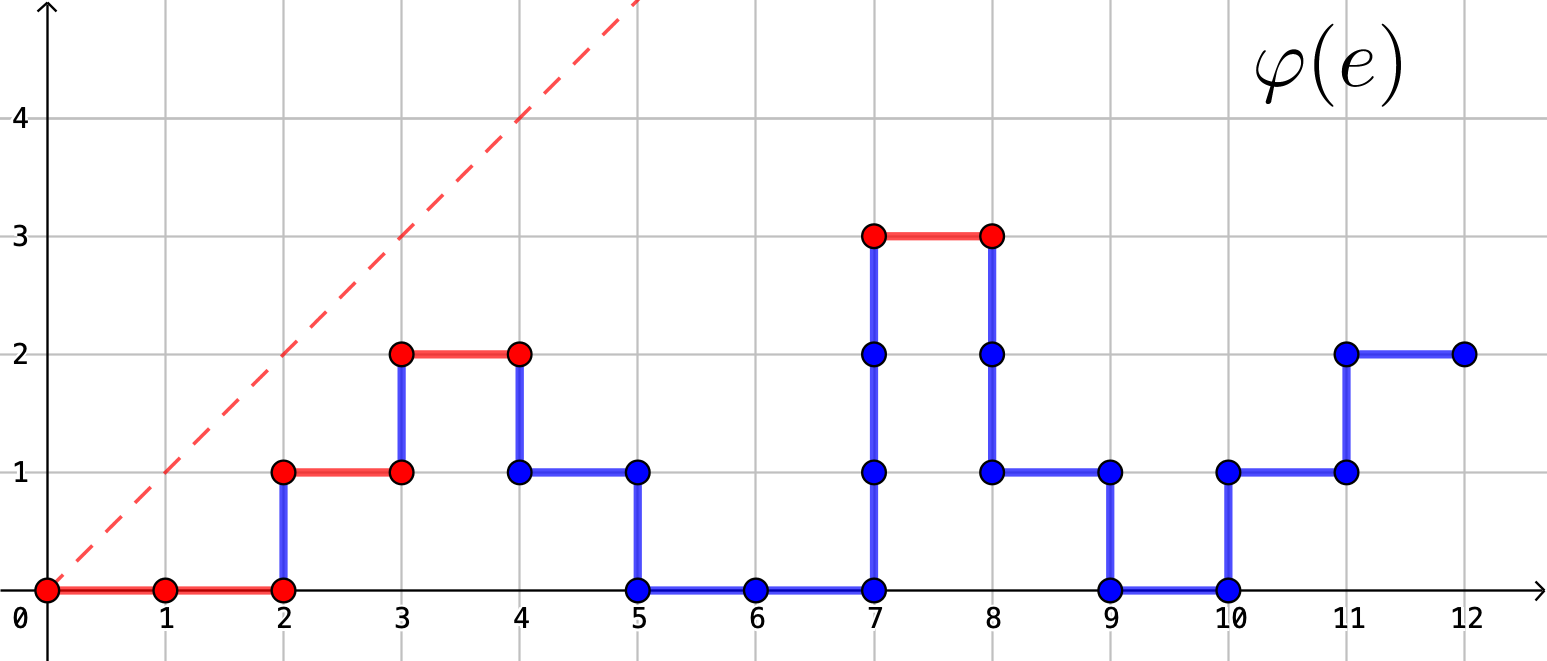}\par
\end{center}
\caption[Example of the action of the bijection $\varphi$]{The inversion sequence ${e=001200132101}$ and its image $\varphi(e)=001210031012$. The weak left-to-right maxima $\W(e)=\W(\varphi(e))=\left\{1,2,3,4,8\right\}$ are drawn in red.}\label{fig:2-01_2-10}
\end{figure}

It is clear by construction that $\varphi$ preserves weak left-to-right maxima, that is, $\W(e)=\W(e')$. It follows that $e'_{i}<i$ for all $i$, and so $e'\in\bI_{n}$. In addition, $\varphi$ is an involution, hence also a bijection. It remains to show that $e\in\bI_{n}\left(2\underline{01}\right)$ if and only if $e'\in\bI_{n}\left(2\underline{01}\right)$, or equivalently, that $e\notin\bI_{n}\left(2\underline{01}\right)$ if and only if $e'\notin\bI_{n}\left(2\underline{10}\right)$.

Suppose that $e\notin\bI_{n}\left(2\underline{01}\right)$, and let $e_{j}e_{i}e_{i+1}$ be an occurrence of $2\underline{01}$. Then $i$ and $i+1$ cannot be weak left-to-right maxima, and so there exists $l$ such that $w_{l}<i<i+1<w_{l+1}$ (with the convention $w_{t+1}:=n+1$). Writing $i=w_{l}+u$, we have
\[
e'_{w_{l+1}-u}=e_{i} \quad\textnormal{ and }\quad e'_{w_{l+1}-u-1}=e_{i+1}.
\]
Since $e'_{w_{l}}=e_{w_{l}}$, we deduce that $e'_{w_{l}}e'_{w_{l+1}-u-1}e'_{w_{l+1}-u}=e_{w_l}e_{i+1}e_{i}$
is an occurrence of $2\underline{10}$ in $e'$, and so $e'\notin\bI_{n}\left(2\underline{10}\right)$.

A similar argument shows that if $e'\notin\bI_{n}\left(2\underline{10}\right)$, then $e\notin\bI_{n}\left(2\underline{01}\right)$.
\end{proof}

It is important to note that the number of occurrences of $2\underline{01}$ in $e$ does not always coincide with the number of occurrences of $2\underline{10}$ in $e'=\varphi(e)$. For instance, ${e=0123012242}$ contains 3 occurrences of $2\underline{01}$ (namely, $e_{3}e_{5}e_{6}$, $e_{4}e_{5}e_{6}$, and $e_{4}e_{6}e_{7}$) but ${e'=0123221042}$ contains 5 occurrences of $2\underline{10}$ (namely, $e'_{3}e'_{7}e'_{8}$, $e'_{4}e'_{6}e'_{7}$, $e'_{4}e'_{7}e'_{8}$, $e'_{5}e'_{7}e'_{8}$, and $e'_{6}e'_{7}e'_{8}$). In fact, there are $470$ inversion sequences of length $7$ containing exactly one occurrence of $2\underline{10}$, but only $466$ containing exactly one occurrence of $2\underline{01}$. Hence, $2\underline{01}$ and $2\underline{10}$ are not strongly Wilf equivalent.

Next we generalize the proof of Proposition~\ref{prop:2_01_and_2_10} to prove Theorem~\ref{thm:gen_2-01_2-10}.

\begin{proof}[Proof of Theorem~\ref{thm:gen_2-01_2-10}]
To prove that $(d{+}1)\underline{p_{1}p_{2}\ldots p_{r}}\sim (d{+}1)\underline{p_{r}p_{r-1}\ldots p_{1}}$, we show that $e\in\bI_n$ contains $(d{+}1)\underline{p_{1}p_{2}\ldots p_{r}}$ if and only if $e'=\varphi(e)$, defined as in Equation~\eqref{eq:eprime}, contains $(d{+}1)\underline{p_{r}p_{r-1}\ldots p_{1}}$.

Suppose that $e_{j}e_{i}e_{i+1}\ldots e_{i+r-1}$ is an occurrence of $(d{+}1)\underline{p_{1}p_{2}\ldots p_{r}}$ in $e$. Since $e_j$ is the largest entry of this occurrence, we know that $i,i+1,\dots,i+r-1$ are not weak left-to-right maxima of $e$. Write $\W(e)=\left\{w_{1},w_{2},\ldots,w_{t}\right\}_<$, and let $l$ be such that $w_{l}<i<i+r-1<w_{l+1}$ (again with the convention $w_{t+1}:=n+1$). Writing $i=w_{l}+u$, we have $e'_{w_{l+1}-u-s}=e_{i+s}$, for $0\leq s\leq r-1$. Thus,
\[
e'_{w_{l}}e'_{w_{l+1}-u-r+1}e'_{w_{l+1}-u-r+2}\ldots e'_{w_{l+1}-u}=e_{w_{l}}e_{i+r+1}e_{i+r-2}\dots e_{i}
\]
is an occurrence of $(d{+}1)\underline{p_{r}p_{r-1}\ldots p_{1}}$ in $e'$, and so $e'\notin\bI_{n}\left((d{+}1)\underline{p_{r}p_{r-1}\ldots p_{1}}\right)$.

A similar argument shows that if $e'\notin\bI_{n}\left((d{+}1)\underline{p_{r}p_{r-1}\ldots p_{1}}\right)$, then $e\notin\bI_{n}\left((d{+}1)\underline{p_{1}p_{2}\ldots p_{r}}\right)$.
\end{proof}

\subsection{The patterns $1\protect\underline{01}$ and $1\protect\underline{10}$}

Next we prove Theorem~\ref{EquivVinc}(iii) using an inclusion-exclusion argument.
For $e\in\bI_{n}$ and $p\in\{1\ul{01},1\ul{10}\}$, define
$$\Em(p,e) = \left\{\{t,i,i+1\}_<:e_{t}e_{i}e_{i+1}\textnormal{ is an occurrence of }p\right\}.$$

\begin{lem}\label{lem:1_01_and_1_10_marked_occ} For every $n$ and $k$, there exists a bijection
\begin{multline}\label{eq:phi}
\phi:\left\{(e,M):e\in\bI_{n},\,M\subseteq\Em\left(1\underline{01},e\right),\,|M|=k\right\}\\
\to \left\{(e',M'):e'\in\bI_{n},\,M'\subseteq\Em\left(1\underline{10},e'\right),\,|M'|=k\right\}.
\end{multline}
\end{lem}

\begin{proof}
For $p\in\{1\ul{01},1\ul{10}\}$, pairs $(e,M)$ where $M\subseteq\Em\left(p,e\right)$ and $|M|=k$ can be interpreted as inversion sequences $e$ with $k$ {\em marked} occurrences of $p$, which are recorded by the set $M$.

Let $(e,M)$ be a pair from the left-hand side of~\eqref{eq:phi}. We will describe its image $\phi((e,M))=(e',M')$. First, let
\begin{equation}\label{eq:proj}
S=\proj(M)=\{i:\{t,i,i+1\}_<\in M\}\subseteq[n]
\end{equation}
be the set of the middle positions of the marked occurrences of $1\ul{01}$, disregarding multiplicities.
Write $S$ uniquely as a disjoint union of consecutive blocks (i.e., maximal subsets whose entries are consecutive), as $S=\bigsqcup_{j=1}^{m}B_{j}$, where $B_{j}=\left\{i_{j},i_{j}+1,\ldots,i_{j}+l_{j}-1\right\}$, with $l_{j}\geq 1$ and $i_j+l_j<i_{j+1}$, for all $j$.

We define $e'$ by setting $e'_i=e_{\rho_S(i)}$ for $1\le i\le n$, where $\rho_S$ is 
defined by
\begin{equation}\label{eq:rhoS}
  \rho_S(i) = \begin{cases}
      2i_{j}+l_{j}-i, & \text{if } i\in B_{j}\cup \{i_{j}+l_{j}\}\text{ for some }j;\\
      i, & \text{otherwise.}
      \end{cases}
\end{equation}

As illustrated in Figure~\ref{fig:1_01_1_10_2}, the transformation $e\mapsto e'$ reverses the entries of $e$ in positions $B_{j}\cup \{i_{j}+l_{j}\}$,
for each $j$, that is, $e'_{i_{j}}e'_{i_{j}+1}\ldots e'_{i_{j}+l_{j}} = e_{i_{j}+l_{j}} \ldots e_{i_{j}+1}e_{i_{j}}$.
Define $$M'=\left\{\{\rho_S(t),\rho_S(i+1),\rho_S(i)\}:\{t,i,i+1\}_<\in M\right\}.$$

\begin{figure}[htb]
	\noindent \begin{centering}
  \includegraphics[width=0.8\textwidth]{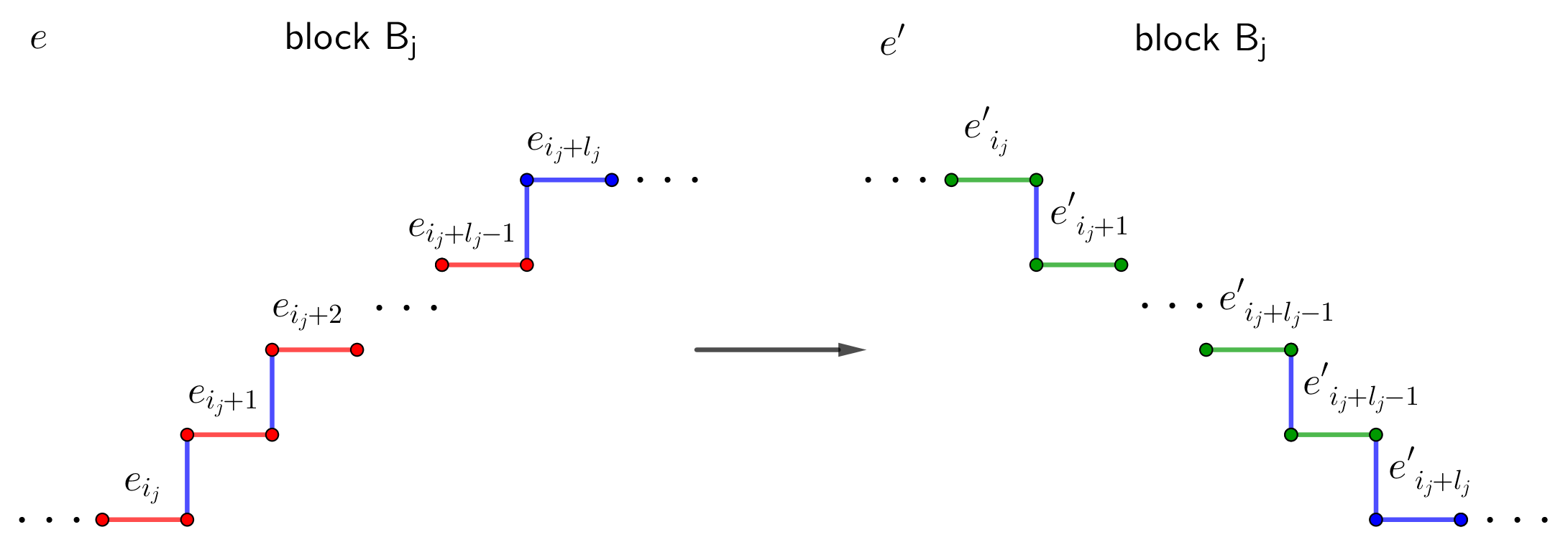}
		\par\end{centering}

  \protect\caption{Schematic diagram of the transformation $e\mapsto e'$ defined in Equation~\eqref{eq:rhoS}.}
	\label{fig:1_01_1_10_2}

\end{figure}

Let us show that $(e',M')$ belongs to the right-hand side of~\eqref{eq:phi}.
For each block $B_j$, since $i_{j}+l_{j}-1\in S$, there exists $t<i_{j}$ such that $e_{t}=e_{i_{j}+l_{j}}$. Thus, if
$i\in B_{j}\cup \{i_{j}+l_{j}\}$,
then
\[
e'_{i} = e_{2i_{j}+l_{j}-i} \leq e_{i_{j}+l_{j}} = e_{t} < t < i_{j} \leq i,
\]
and so $e'\in\bI_n$. Now we argue that $M'\subseteq\Em\left(1\underline{10},e'\right)$. For every $\{t,i,i+1\}_<\in M$, if $B_j$ is the block that $i$ belongs to, then $t<i_j$, and so
$e'_{\rho_S(t)}e'_{\rho_S(i+1)}e'_{\rho_S(i)}=e_te_{i+1}e_i$ is an occurrence of $1\underline{10}$ in $e'$.
It is also clear by construction that $|M'|=|M|=k$.

Finally, the fact that $\proj(M')=S$ allows us to describe the inverse of the map $\phi$ as follows. Given a pair $(e',M')$ from the right-hand side of~\eqref{eq:phi}, let $S=\proj(M')$. Let $\phi'((e',M'))$ be the pair $(e,M)$ obtained by setting $e_i=e'_{\rho_S(i)}$ for $1\le i\le n$ and $$M'=\left\{\{\rho_S(t'),\rho_S(i'+1),\rho_S(i')\}:\{t',i',i'+1\}_<\in M'\right\}.$$
The fact that $\rho_S:[n]\to[n]$ is an involution implies that $\phi$ and $\phi'$ are inverses of each other.
\end{proof}

\begin{prop}\label{prop:1_01_and_1_10} The patterns $1\underline{01}$ and $1\underline{10}$ are strongly Wilf equivalent.
\end{prop}

\begin{proof}
For $p\in\{1\ul{01},1\ul{10}\}$, let
$$\mu_n(p,k)=\left|\left\{(e,M):e\in\bI_{n},\,M\subseteq\Em\left(p,e\right),\,|M|=k\right\}\right|$$
be the number of inversion sequences in $\bI_n$ with $k$ marked occurrences of $p$.

An inversion sequence with $k$ marked occurrences of $p$ can be constructed by first choosing an inversion sequence with exactly $m$ occurrences of $p$, for some $m\ge k$, and then marking $k$ occurrences, which can be done in $\binom{m}{k}$ ways. It follows that
\begin{equation}\label{eq:mu}
\mu_n(p,k)=\sum_{m\ge k} \left|\bI_n(p,m)\right|\binom{m}{k},
\end{equation}
which can be inverted using the Principle of Inclusion-Exclusion to obtain an expression for $\left|\bI_n(p,m)\right|$ in terms of $\mu_n(p,k)$.

Specifically, multiplying Equation~\eqref{eq:mu} by $x^k$ and summing over $k\ge0$, we can write it as an equality of polynomials:
$$\sum_{k\ge0} \mu_n(p,k)x^k=\sum_{k\ge0} \sum_{m\ge k} \left|\bI_n(p,m)\right|\binom{m}{k}x^k=\sum_{m\ge0} \left|\bI_n(p,m)\right| \sum_{k=0}^m \binom{m}{k} x^k= \sum_{m\ge0} \left|\bI_n(p,m)\right|(1+x)^m.$$
Setting $x=y-1$, we get
$$\sum_{m\ge0} \left|\bI_n(p,m)\right| y^m=\sum_{k\ge0} \mu_n(p,k)(y-1)^k=\sum_{k\ge0}\sum_{m=0}^k \mu_n(p,k) \binom{k}{m}(-1)^{k-m} y^m,$$
and taking coefficients of $y^m$ on both sides,
$$\left|\bI_n(p,m)\right|=\sum_{k\ge m}\mu_n(p,k) \binom{k}{m}(-1)^{k-m}.$$
By Lemma~\ref{lem:1_01_and_1_10_marked_occ}, $\mu_n(1\ul{01},k)=\mu_n(1\ul{10},k)$ for all $n$ and $k$. We conclude that
$$\left|\bI_n(1\ul{01},m)\right|=\left|\bI_n(1\ul{10},m)\right|$$
for all $n$ and~$m$.
\end{proof}

Even though Proposition~\ref{prop:1_01_and_1_10} states that the number of occurrences $1\underline{01}$ is equidistributed with the number of occurrences of $1\underline{10}$ over $e\in\bI_n$, the joint distribution of the number of occurrences of these two patterns is not symmetric, that is,  there exist integers $l,m,n$ such that
\[
\left|\bI_{n}\left(1\underline{01},m\right)\cap\bI_{n}\left(1\underline{10},l\right)\right| \neq \left|\bI_{n}\left(1\underline{01},l\right)\cap\bI_{n}\left(1\underline{10},m\right)\right|.
\]
For instance, $\bI_{6}\left(1\underline{01},3\right)\cap\bI_{6}\left(1\underline{10},0\right) = \emptyset$, but $\bI_{6}\left(1\underline{01},0\right)\cap\bI_{6}\left(1\underline{10},3\right) = {\{011110\}}$.

It is also easy to check that the patterns $1\underline{01}$ and $1\underline{10}$ are not super-strongly Wilf equivalent. Indeed, with the notation from Equation~\eqref{eq:def_super_strongly}, the sets
\begin{align*}
\bI_{6}\left(1\underline{01},\left\{\{2,5,6\},\{4,5,6\}\right\}\right) &= \left\{{012101}\right\} \textrm{ and }\\
\bI_{6}\left(1\underline{10},\left\{\{2,5,6\},\{4,5,6\}\right\}\right) &= \left\{{010110}, {012110}\right\}
\end{align*}
have different cardinalities.

We end this section showing that, despite $1\underline{01}$ and $1\underline{10}$ not being super-strongly Wilf equivalent, the sets of positions of the  middle entries of occurrences of these patterns in inversion sequences are equidistributed. This is stated as Proposition~\ref{prop:1_01_and_1_10_S} below, and proved using an inclusion-exclusion argument, similar to the one used to prove the equivalence $\underline{110}\stackrel{ss}{\sim} \underline{100}$ in~\cite[Prop.~3.11]{AuliElizalde}.

For $e\in\bI_{n}$ and $p\in\{1\ul{01},1\ul{10}\}$, and letting $\proj$ be defined as in Equation~\eqref{eq:proj}, define
$$\Em^{*}(p,e) = \proj(\Em(p,e))=\{i:\exists j<i\textnormal{ such that }e_{j}e_{i}e_{i+1}\textnormal{ is an occurrence of }p\}.$$

\begin{lem}\label{lem:1_01_and_1_10} For every $S\subseteq [n]$, the map $e\mapsto e'$ where $e'_i=e_{\rho_S(i)}$ for $1\le i\le n$, as defined in Equation~\eqref{eq:rhoS}, is a bijection
$$
\left\{e\in\bI_{n}:\Em^{*}\left(1\underline{01},e\right)\supseteq S\right\}\rightarrow \left\{e'\in\bI_{n}:\Em^{*}\left(1\underline{10},e'\right)\supseteq S\right\}.
$$
\end{lem}

\begin{proof}
Given $e\in\bI_{n}$ with $\Em^{*}\left(1\underline{01},e\right)\supseteq S$, the same argument as in the proof of Lemma~\ref{lem:1_01_and_1_10_marked_occ} shows that its image $e'$ satisfies that $e'\in\bI_n$ and $\Em^{*}\left(1\underline{10},e'\right)\supseteq S$.
This map is a bijection because, for any $e'\in\bI_n$ with $\Em^{*}\left(1\underline{10},e'\right)\supseteq S$, one can recover its preimage $e$ by setting $e_i=e'_{\rho_S(i)}$ for $1\le i\le n$.
\end{proof}

\begin{prop}\label{prop:1_01_and_1_10_S} For every $S\subseteq [n]$,
\[
\left|\left\{e\in\bI_{n}:\Em^{*}\left(1\underline{01},e\right)= S\right\}\right| = \left|\left\{e\in\bI_{n}:\Em^{*}\left(1\underline{10},e\right)= S\right\}\right|.
\]
\end{prop}

\begin{proof}
For $p\in\{1\ul{01},1\ul{10}\}$ and $S\subseteq [n]$, let
\[
f^p_{=}(S)=\left|\left\{e\in\bI_{n}:\Em^{*}(p,e)= S\right\}\right|\quad\textnormal{ and }\quad f^p_{\geq}(S)=\left|\left\{e\in\bI_{n}:\Em^{*}(p,e)\supseteq S\right\}\right|.
\]
It is clear that, for every $S\subseteq [n]$,
\[
\left\{e\in\bI_{n}:\Em^{*}\left(p,e\right)\supseteq S\right\} = \bigsqcup_{T\supseteq S}\left\{e\in\bI_{n}:\Em^{*}(p,e)= T\right\},
\]
and so
\[
f^p_{\geq}(S)=\sum_{T\supseteq S}f^p_{=}(T).
\]
The Principle of Inclusion-Exclusion~{\cite[Thm.\ 2.1.1]{Stanley}} implies that
$$
f^p_{=}(S)=\sum_{T\supseteq S}(-1)^{|T\setminus S|}f^p_{\geq}(T).
$$
By Lemma~\ref{lem:1_01_and_1_10}, $f^{1\ul{01}}_{\geq}(T)=f^{1\ul{10}}_{\geq}(T)$, for all $T\subseteq[n]$. Thus, it follows that
$f^{1\ul{01}}_{=}(S)=f^{1\ul{10}}_{=}(S)$, for all $S\subseteq[n]$.
\end{proof}

\subsection{The patterns $\protect\underline{10}0$ and $\protect\underline{10}1$}\label{subsec:10-0_10-1}

In this subsection we prove Theorem~\ref{EquivVinc}(ii). Our approach relies on constructing isomorphic generating trees for inversion sequences avoiding $\underline{10}0$ and for those avoiding $\underline{10}1$. We determine such generating trees by succession rules that describe their growth by insertions on the right, in the same manner that generating trees for certain subclasses of pattern avoiding permutations have been constructed in~\cite{Bernini, Bouvel, ElizaldeTrees}. First, we introduce some terminology regarding generating trees and succession rules, following Bouvel et al.~\cite{Bouvel}. For a more detailed presentation of these topics, see~\cite{Banderier, Barcucci, BousquetMelou, West}.

Let $\mathcal{G}$ be a combinatorial class, with a finite number of objects of size $n$ for each $n\ge0$, and suppose that $\mathcal{G}$ contains exactly one object of size 0. A {\it generating tree} for $\mathcal{G}$ is a (typically infinite) rooted tree whose vertices are the objects of $\mathcal{G}$, and such that objects of size $n$ are at level $n$ in the tree, i.e., at distance $n$ from the root.

\begin{figure}[htb]
\begin{center}
\includegraphics[scale=0.42]{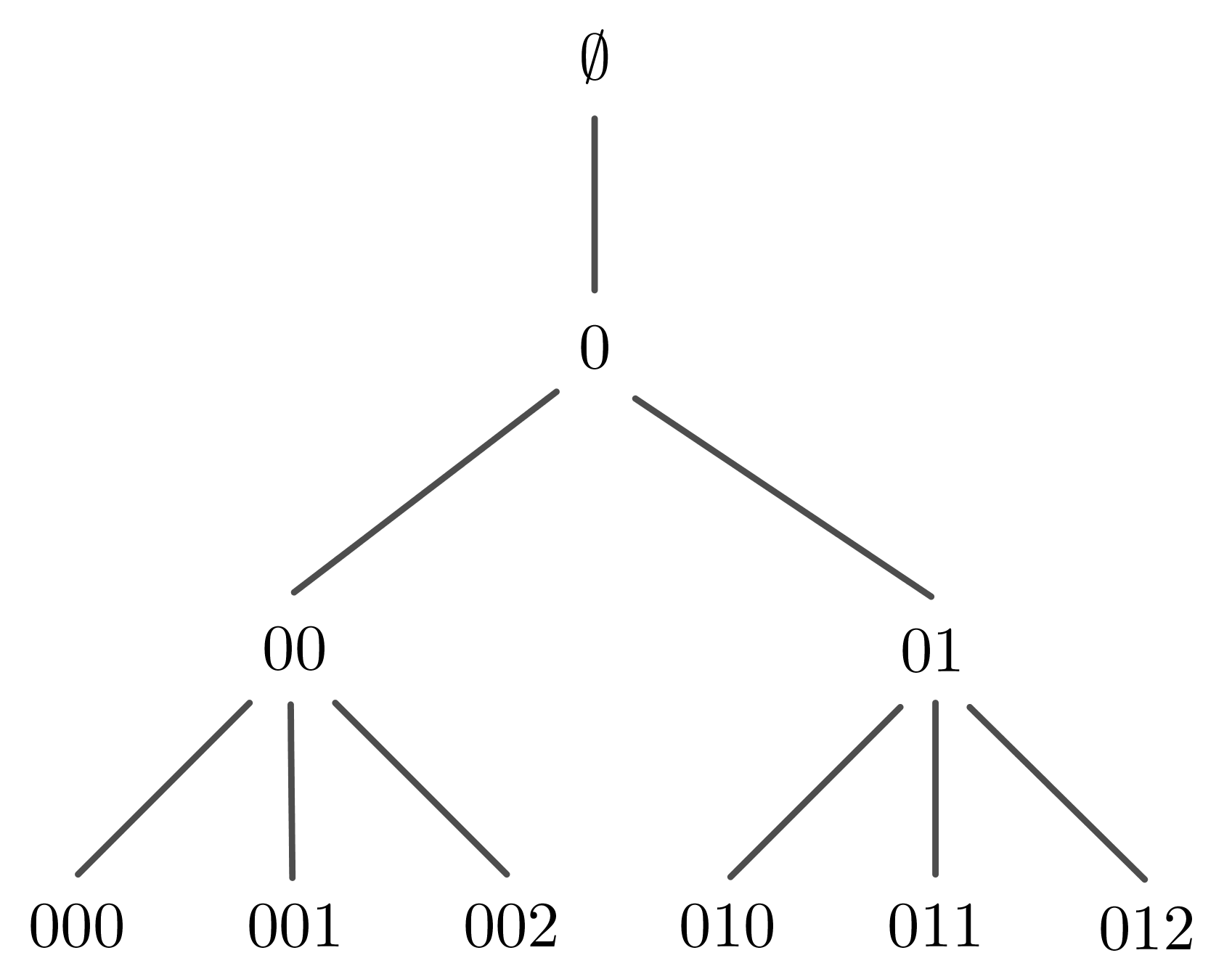}
\hspace*{1.5cm}
\includegraphics[scale=0.42]{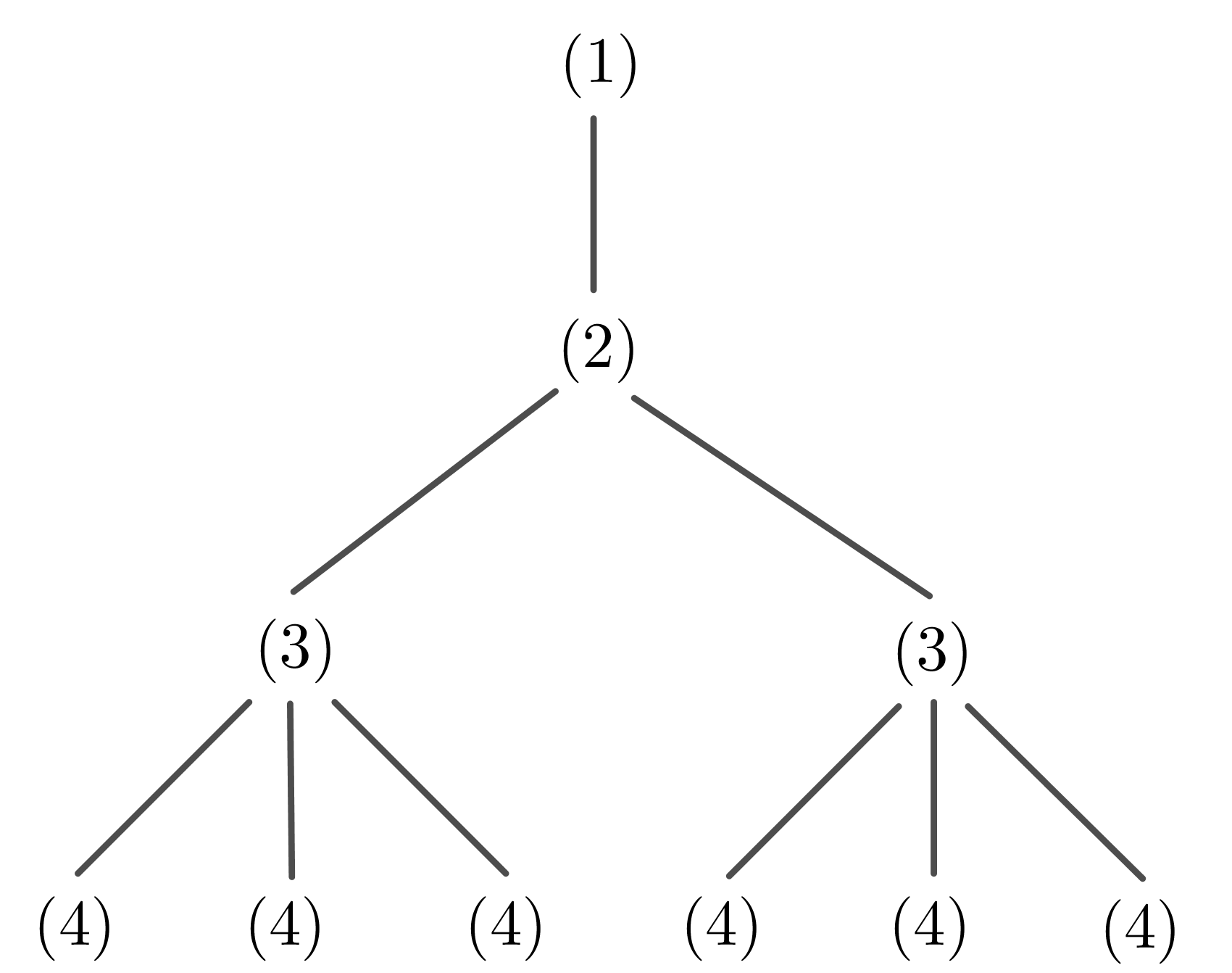}\par
\end{center}
\caption[Two representations of a generating tree for inversion sequences]{Two representations of a generating tree for inversion sequences: with objects for vertices (left) and with labels from the succession rule $\Omega_{\Inv}$ for vertices (right).}\label{fig:figure_trees}
\end{figure}

The children of an object $g\in\mathcal{G}$ are obtained by adding an {\it atom} ---that is, a piece that increases the size by 1--- to $g$. These additions must follow certain prescribed rules, which are determined by the structure of the objects of $\mathcal{G}$. In particular, these rules ensure that each object appears exactly once in the tree. We refer to the process of adding an atom to $g\in\mathcal{G}$ as the {\it growth} of $g$.

Given an inversion sequence $e=e_{1}e_{2}\ldots e_{n}\in\bI_{n}$, we grow $e$ by inserting an entry $h$ on the right, chosen from the set of values $\left\{0,1,\ldots,n\right\}$, called {\it sites}, to obtain the inversion sequence
\[
eh=e_{1}e_{2}\ldots e_{n}h\in\bI_{n+1}.
\]
The generating tree obtained in this manner, depicted in Figure~\ref{fig:figure_trees}(left), is the one for the class of all inversion sequences, which we denote by $\Inv=\bigcup_{n\ge0}\bI_n$. If instead we consider the subclass of inversion sequences satisfying a certain restriction, then not all the sites in $\left\{0,1,\ldots,n\right\}$ are valid choices for $h$, in the sense that $eh$ may not belong to the subclass. Sites that are valid are called the {\it active sites} of $e$. We denote the subclass of inversion sequences avoiding the pattern $p$ by $\Inv(p)=\bigcup_{n\ge0}\bI_n(p)$.
Whenever we speak of the growth or the active sites of $e\in\bI_{n}(p)$, we think of $e$ as an object of the class $\Inv(p)$, as opposed to as an object of $\Inv$.

\begin{exa}
The active sites of $e={01032453}\in\bI_8(\protect\underline{10}0)$ are $\{1,4,5,6,7,8\}$, since these are the values of $h$ for which $eh\in\bI_{9}(\underline{10}0)$, see Figure~\ref{fig:active_sites}.
\end{exa}

\begin{figure}[htp]
\begin{center}
	\includegraphics[scale=0.52]{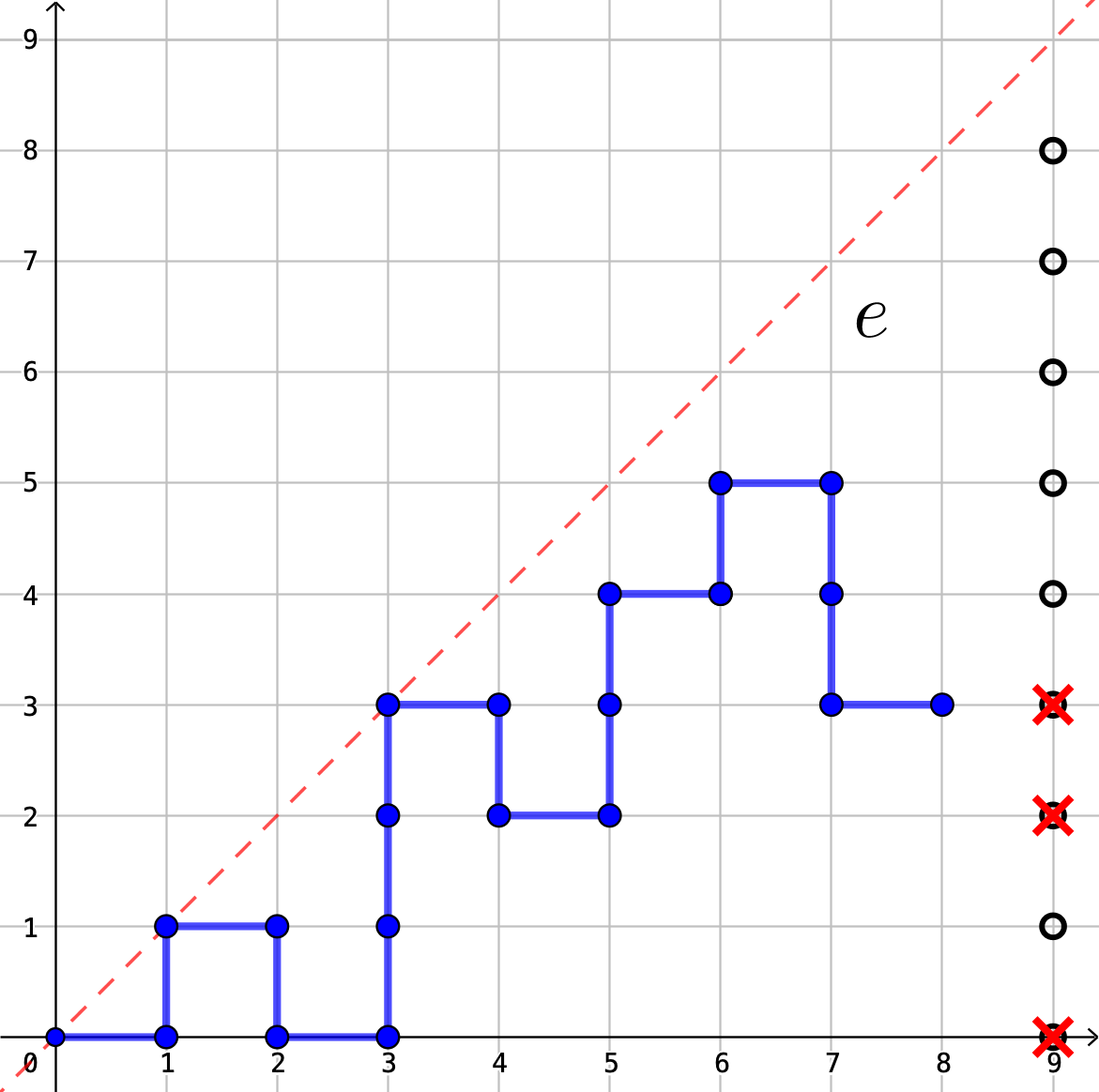}
\end{center}
\caption[Growth of $e={01032453}\in\bI_{8}(\protect\underline{10}0)$ by insertions on the right]{Growth of $e={01032453}\in\bI_{8}(\protect\underline{10}0)$ by insertions on the right. Each site $h$ of $e$ is represented by a circle at $(9,h)$. The active sites are $\{1,4,5,6,7,8\}$, whereas inactive sites $0$, $2$, and $3$ are crossed out.}\label{fig:active_sites}
\end{figure}

Given $e\in\bI_{n}$, we say that $i$ is a {\it descent} of $e$ if $e_{i}>e_{i+1}$, and let
$\Des(e)=\{i\in[n-1]:e_{i}>e_{i+1}\}$ denote its descent set.
Let us show that the active sites of an inversion sequence in $\Inv(\underline{10}0)$ or $\Inv(\underline{10}1)$ are determined by its descents.

\begin{lem}\label{lem:10-0_and_10-1}
The active sites of $e\in\bI_n(\underline{10}0)$ are
$$\{0,1,\dots,n\}\setminus\{e_{i+1}:i\in\Des(e)\}.$$
The active sites of $e\in\bI_n(\underline{10}1)$ are
$$\{0,1,\dots,n\}\setminus\{e_{i}:i\in\Des(e)\}.$$
In particular, $e_n$ is an active site of $e\in\bI_n(\underline{10}1)$.
\end{lem}

\begin{proof} A value $h\in\{0,1,\dots,n\}$ is an active site of $e\in\bI_n(\underline{10}0)$ if and only if inserting $h$ on the right of $e$ does not create an occurrence of $\underline{10}0$, that is, if there does not exist $i<n$ such that $e_i>e_{i+1}=h$.
Similarly, $h\in\{0,1,\dots,n\}$ is an active site of $e\in\bI_n(\underline{10}1)$ if and only if inserting $h$ on the right of $e$ does not create an occurrence of $\underline{10}1$, that is, if there does not exist $i<n$ such that $h=e_i>e_{i+1}$.

Finally, suppose for the sake of contradiction that $e_n$ is not an active site of $e\in\bI_{n}\left(\underline{10}1\right)$. Then there must exist $i\in\Des(e)$ such that $e_{i}=e_{n}$. Since $e_{n}=e_{i}>e_{i+1}$, we must have $i+1<n$, and so ${e_{i}e_{i+1}e_{n}}$ would be an occurrence of $\underline{10}1$.
\end{proof}

A {\it succession rule} describes a generating tree by identifying its vertices with {\it labels}. It provides a label for the root, and an inductive rule to produce the labels of the children given the label of the parent. For example, assigning to each inversion sequence $e$ the label $(a)$, where $a$ is the number of active sites of $e$, yields the following succession rule for the generating tree for $\Inv$:
\[
  \Omega_{\Inv} = \begin{cases}
      (1), & \\
      (a)\rightsquigarrow & \hspace{-3mm}(a+1)^{a}.
      \end{cases}
\]
This rule means that the root, which is the empty inversion sequence, has label $(1)$, and that every object with label $(a)$ has $a$ children, each with label $(a+1)$. Figure~\ref{fig:figure_trees}(right) shows these labels on the generating tree for $\Inv$.

If two generating trees have the same succession rule, then they have the same number of vertices at level $n$, for each $n$. To prove that $\left|\bI_{n}\left(\underline{10}0\right)\right|=\left|\bI_{n}\left(\underline{10}1\right)\right|$, we will show that $\Inv\left(\underline{10}0\right)$ and $\Inv\left(\underline{10}1\right)$ have generating trees with the same succession rule.

\begin{prop}\label{prop:10-0_and_10-1_1} The class $\Inv(\underline{10}0)$ has a generating tree described by the succession rule
\[
  \Omega_{\Inv\left(\underline{10}0\right)} = \begin{cases}
      (1,0), & \\
      (a,b)\rightsquigarrow & \hspace{-3mm}(a+1,b),(a,b+1),\ldots,(2,b+a-1), \\
        & \hspace{-3mm}(a+b,0),(a+b-1,1),\ldots,(a+1,b-1).
      \end{cases}
\]
\end{prop}

\begin{proof}
We construct a generating tree by insertions on the right. To each $e\in\bI_{n}\left(\underline{10}0\right)$, we assign the label $(a,b) = \left(\left|A_{\geq}(e)\right|,\left|B_{<}(e)\right|\right)$, where
\begin{align}\label{eq:def_sets_sites}
\begin{split}
A_{\geq}(e) &= \left\{h:h\textnormal{ is an active site of }e\textnormal{ and }h\geq e_{n}\right\},\\
B_{<}(e) &= \left\{h:h\textnormal{ is an active site of }e\textnormal{ and }h<e_{n}\right\},
\end{split}
\end{align}
with the convention $e_0=0$. The root, which is the empty inversion sequence, has label $(1,0)$.

Suppose that $e\in\bI_{n}\left(\underline{10}0\right)$ has label $(a,b)$, and that we grow $e$ by inserting $h$ on the right, obtaining $eh\in\bI_{n+1}\left(\underline{10}0\right)$. Since $h$ is an active site of $e$, it belongs to either $A_{\geq}(e)$ or $B_{<}(e)$.

Suppose first that $h\in A_{\geq}(e)$, and that $h$ is the $i$th smallest element in $A_{\geq}(e)$. Since $n$ is not a descent of $eh$, all the active sites of $e$ are also active in $eh$, by Lemma~\ref{lem:10-0_and_10-1}, and there is an additional active site $n+1$. Thus, $eh$ has label
\[
((a-i+1)+1, b+(i-1)) = (a+2-i,b-1+i),
\]
as illustrated in Figure~\ref{fig:succession}(top).
As $i$ ranges from $1$ to $a$, the resulting inversion sequences $eh\in \bI_{n+1}\left(\underline{10}0\right)$ where $h\in A_{\geq}(e)$ have labels
\[
(a+1,b),(a,b+1),\ldots,(2,b+a-1).
\]

\begin{figure}[htb]
\begin{center}
\includegraphics[scale=0.52]{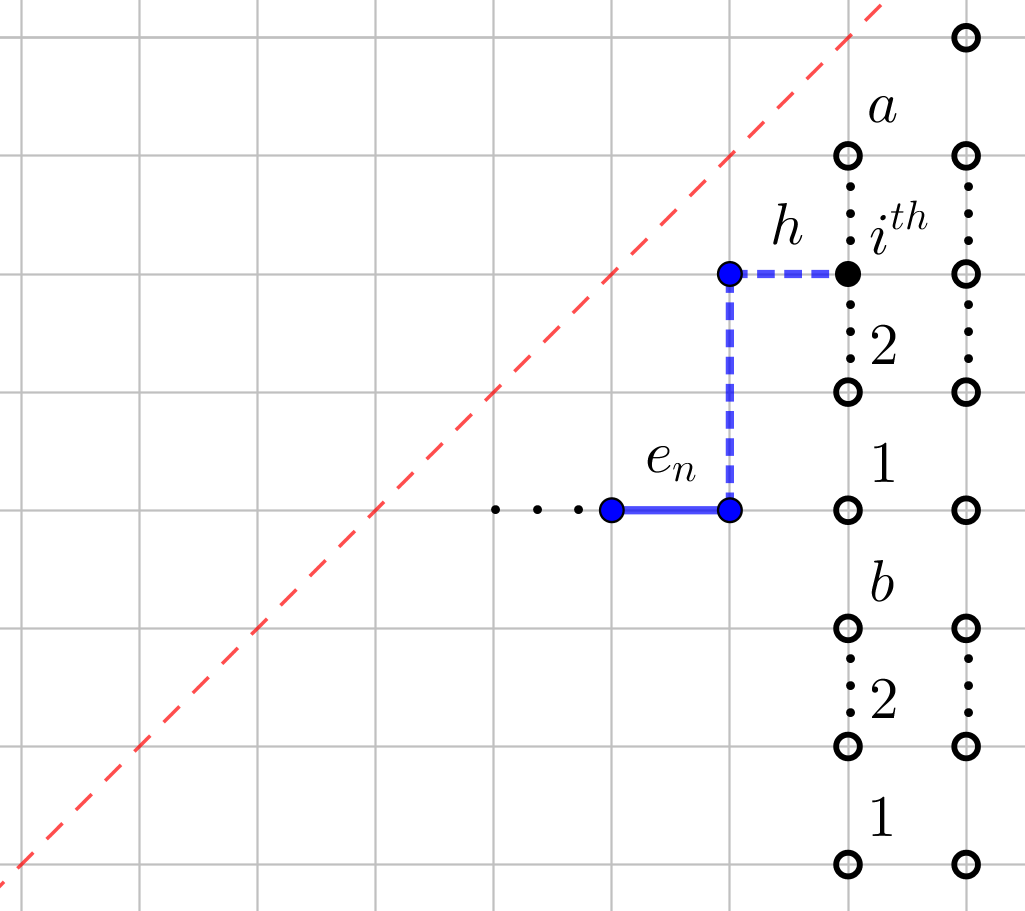}
\hspace*{1.5cm}
\includegraphics[scale=0.52]{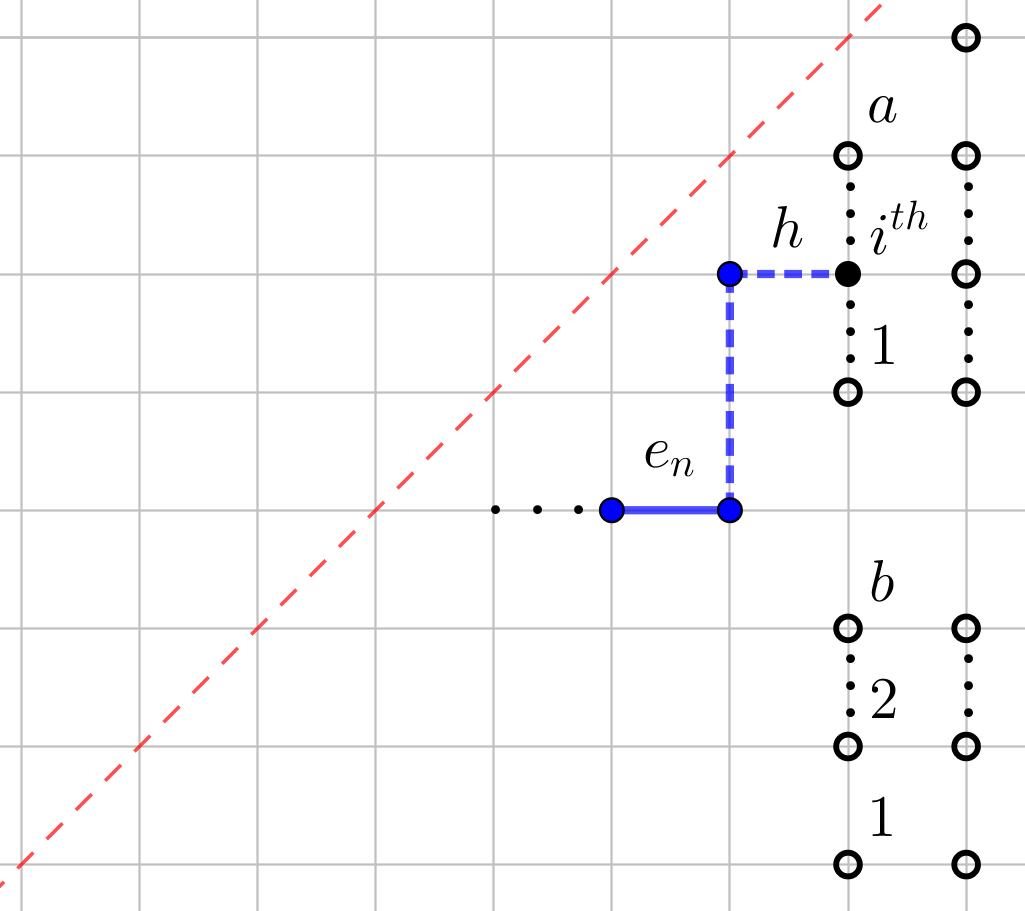}
\bigskip

\includegraphics[scale=0.52]{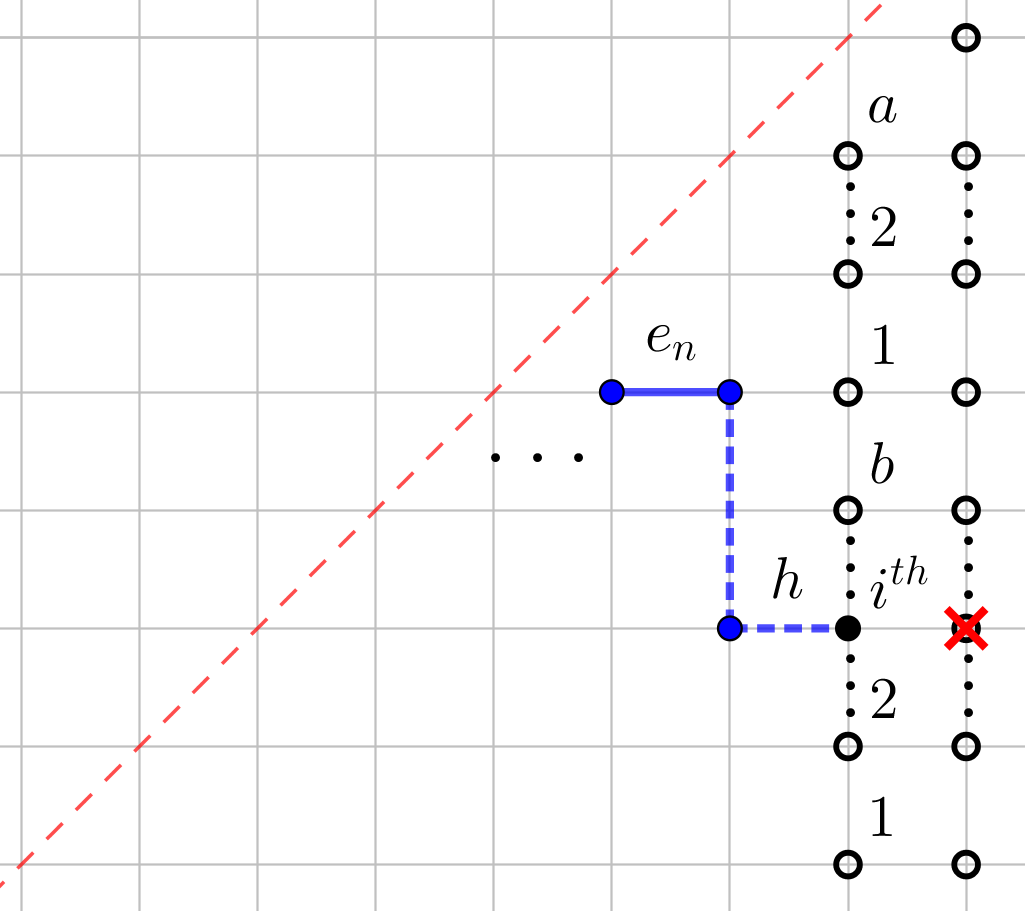}
\hspace*{1.5cm}
\includegraphics[scale=0.52]{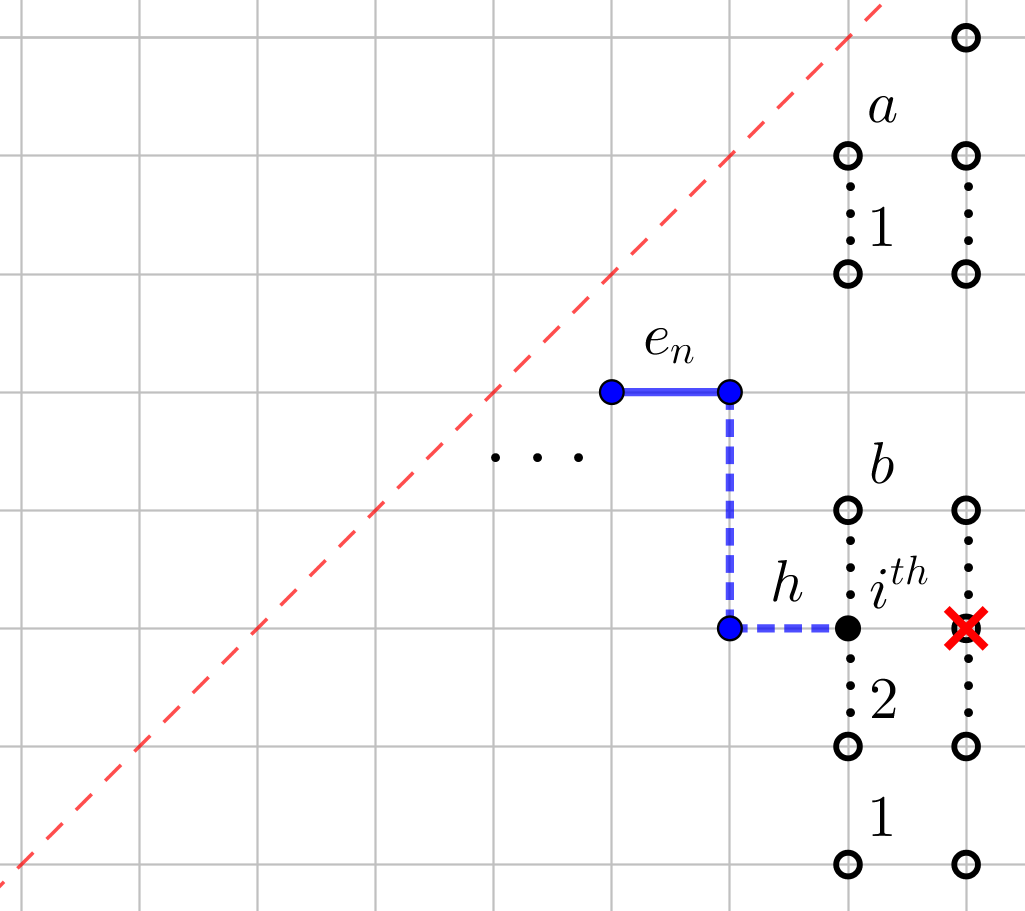}\par
\end{center}
\caption[Growth of $e\in\bI_{n}(\protect\underline{10}0)$]{Growth of $e\in\bI_{n}(\protect\underline{10}0)$ when the $i$th smallest active site in $A_{\geq}(e)$ (top) or $B_{<}(e)$ (bottom) is chosen. The site $e_{n}$ may be active (left) or not (right), but this does not affect the label of $eh$.}
\label{fig:succession}
\end{figure}

Suppose now that $h\in B_{<}(e)$, and that $h$ is the $i$th smallest element in $B_{<}(e)$.
Then $n$ is a descent of $eh$, so Lemma~\ref{lem:10-0_and_10-1} implies that $h$ is not an active site of $eh$. With the additional active site $n+1$, the inversion sequence $eh$ has label
\[
((b-i)+a+1, i-1) = (a+b+1-i,i-1),
\]
as shown in Figure~\ref{fig:succession}(bottom). As $i$ ranges from $1$ to $b$, the resulting inversion sequences $eh\in \bI_{n+1}\left(\underline{10}0\right)$ where $h\in B_{<}(e)$ have labels
\[
(a+b,0),(a+b-1,1),\ldots,(a+1,b-1).\qedhere
\]
\end{proof}

Next we find a generating tree for $\Inv\left(\underline{10}1\right)$ that is isomorphic to the one described in Proposition~\ref{prop:10-0_and_10-1_1}.

\begin{prop}\label{prop:10-0_and_10-1_2} The class $\Inv(\underline{10}1)$ has a generating tree described by the succession rule
\[
  \Omega_{\Inv\left(\underline{10}1\right)} = \begin{cases}
      (1,0), & \\
      (a,b)\rightsquigarrow & \hspace{-3mm}(a+1,b),(a,b+1),\ldots,(2,b+a-1), \\
       & \hspace{-3mm}(a+b,0),(a+b-1,1),\ldots,(a+1,b-1).
      \end{cases}
\]
\end{prop}

\begin{proof} As in the proof of Proposition~\ref{prop:10-0_and_10-1_1}, we assign, to each $e\in\bI_{n}\left(\underline{10}1\right)$, the label $(a,b) = \left(\left|A_{\geq}(e)\right|,\left|B_{<}(e)\right|\right)$, where $A_{\geq}(e)$ and $B_{<}(e)$ are as in Equation~\eqref{eq:def_sets_sites}. With the convention $e_0=0$, the root again has label $(1,0)$.

Suppose that $e\in\bI_{n}\left(\underline{10}1\right)$ has label $(a,b)$, and that we grow $e$ by inserting $h$ on the right. If the chosen active site $h$ is in $A_{\geq}(e)$, then all the active sites of $e$ are also active in $eh$, by Lemma~\ref{lem:10-0_and_10-1}, and we deduce that the resulting inversion sequences in $\bI_{n+1}\left(\underline{10}1\right)$ have labels
\[
(a+1,b),(a,b+1),\ldots,(2,b+a-1).
\]
The visual representation corresponds again to Figure~\ref{fig:succession}(top left), since $e_{n}$ is an active site of $e$ by Lemma~\ref{lem:10-0_and_10-1}.

Suppose now that $h\in B_{<}(e)$, and that $h$ is the $i$th smallest element in $B_{<}(e)$.
Since $n$ is a descent of $eh$, Lemma~\ref{lem:10-0_and_10-1} implies that $e_{n}$ is not an active site of $eh$, and also that $e_{n}$ was an active site of $e$. In this case, $eh$ has $(a-1)+1+(b-i)+1$ active sites $h'$ such that $h\leq h'$, namely the $a-1$ sites such that $e_{n}\leq h'\leq n$, the site $n+1$, and the $(b-i)+1$ sites such that $h\leq h'< e_{n}$. Hence, $eh$ has label $(a+b+1-i,i-1)$, see Figure~\ref{fig:succession_3}.
As $i$ ranges from $1$ to $b$, the resulting inversion sequences $eh\in \bI_{n+1}\left(\underline{10}1\right)$, where $h\in B_{<}(e)$, have labels
\[
(a+b,0),(a+b-1,1),\ldots,(a+1,b-1).\qedhere
\]
\end{proof}

\begin{figure}[htp]
\begin{center}
	\includegraphics[scale=0.52]{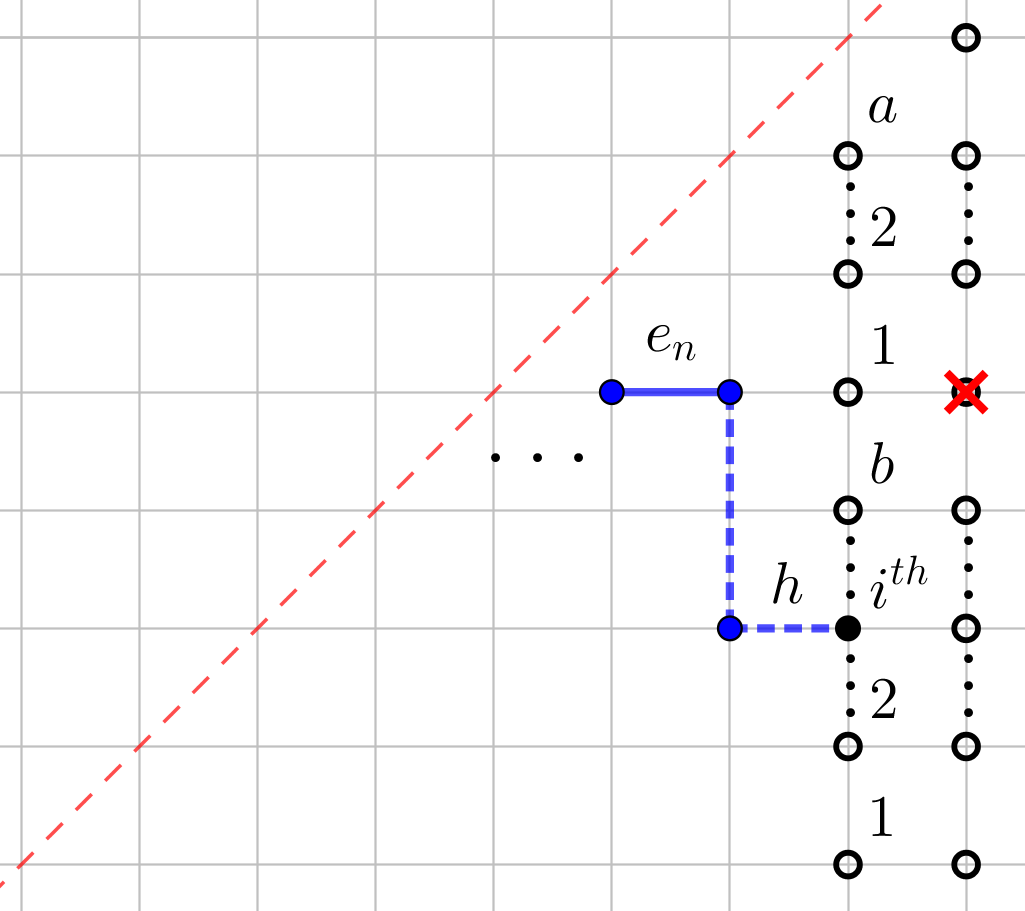}
\end{center}
\caption[Growth of $e\in\bI_{n}(\protect\underline{10}1)$]{Growth of $e\in\bI_{n}(\protect\underline{10}1)$ when the $i$th smallest active site in $B_{<}(e)$ is chosen.}\label{fig:succession_3}
\end{figure}

Since the generating trees for $\Inv\left(\underline{10}0\right)$ and $\Inv\left(\underline{10}1\right)$ described in Propositions~\ref{prop:10-0_and_10-1_1} and~\ref{prop:10-0_and_10-1_2}, respectively, are isomorphic, the next result follows.

\begin{cor}\label{cor:10_0_and_10_1} The patterns $\underline{10}0$ and $\underline{10}1$ are Wilf equivalent.
\end{cor}

We remark that these two patterns are not strongly Wilf equivalent. For instance, the are 134 inversion sequences of length 6 containing exactly one occurrence of $\underline{10}0$, but only 132 containing exactly one occurrence of $\underline{10}1$.

To end this subsection, we use the generating trees from Propositions~\ref{prop:10-0_and_10-1_1} and~\ref{prop:10-0_and_10-1_2} to provide an expression for the generating function
$$A(z)=\sum_{n\ge0} \left|\bI_n\left(\underline{10}0\right)\right| z^n =\sum_{n\ge0} \left|\bI_n\left(\underline{10}1\right)\right| z^n.$$

\begin{prop}\label{prop:GF}
We have that $A(z)=G(1,z)$, where $G(u,z)$ is defined recursively by
\begin{equation}\label{eq:G}
G(u,z)=u(1-u)+uG(u(1+z-uz),z).
\end{equation}
\end{prop}

\begin{proof}
Let $F(u,v,z)$ be the generating function where the coefficient of $u^av^bz^n$ is the number of vertices with label $(a,b)$ at level $n$ of the generating tree with succession rule $\Omega_{\Inv\left(\underline{10}0\right)}$. Note that $A(z)=F(1,1,z)$.
Each term $u^av^b$ corresponding to a label $(a,b)$ at level $n$ of the tree generates a contribution
\begin{align*}&\left(u^{a+1}v^b+u^av^{b+1}+\dots+u^2v^{b+a-1}\right)+\left(u^{a+b}+u^{a+b-1}v+\dots+u^{a+1}v^{b-1}\right)\\
&=\frac{u}{u-v}\left(u^{a+1}v^b-uv^{a+b}+u^{a+b}-u^av^b\right)
\end{align*}
at level $n+1$. This translates into a functional equation for $F(u,v):=F(u,v,z)$, namely
$$F(u,v)=u+\frac{uz}{u-v}\left(uF(u,v)-uF(v,v)+F(u,u)-F(u,v)\right).$$
Letting $G(u)=F(u,u)$ and collecting all the terms with $F(u,v)$ on the left hand side, we get
$$\frac{u-v-(u-1)uz}{u}\, F(u,v)=u-v+zG(u)-uzG(v).$$
The kernel of this equation is canceled by setting $v=u(1+z-uz)$, which gives
$$0=u-u(1+z-uz)+zG(u)-uzG(u(1+z-uz)),$$ or equivalently,
\[
G(u)=u(1-u)+uG(u(1+z-uz)).
\qedhere
\]
\end{proof}

Equation~\eqref{eq:G} can be used to compute the expansion of $G(u,z)$ as a series in the variable $u$. Defining $V^i:=V^i(u,z)$  recursively by $V^0(u,z)=u$ and $V^i(u,z)=V^{i-1}(u(1+z-uz),z)$ for $i\ge1$, we obtain
\begin{align*}G(u,z)&=\sum_{k\ge0} V^0 V^1 \cdots V^k (1-V^k)=u+\sum_{k\ge0} V^0 V^1 \cdots V^k (V^{k+1}-V^k)\\
&=u + z u^2 + (z + 2)z^2u^3 + (z^3 + 5z^2 + 9z+ 5)z^3 u^4 \\
& \quad + (z^6 + 9z^5 + 35z^4 + 75z^3 + 92z^2 + 59z + 14)z^4u^5
+\cdots
\end{align*}
In fact, if follows from Lemma~\ref{lem:10-0_and_10-1} that if a vertex at level $n$ has $k=a+b$ active sites, then $k-1\le n\le\binom{k}{2}$, and so any the exponents of any term $u^kz^n$ with nonzero coefficient in $G(u,z)$ must satisfy this constraint. In particular, the first $k$ terms of the expansion of $G(u,z)$ as a series in $u$ contain the first $k-1$ terms of its expansion as a series in $z$:
$$G(u,z)=u + u^2z + 2u^3z^2 + (5u^4 + u^3)z^3 + (14u^5 + 9u^4)z^4+\cdots$$

\subsection{The patterns $\protect\underline{01}0$ and $\protect\underline{01}1$}

Next we prove Theorem~\ref{EquivVinc}(i). Using ideas similar to those in the previous subsection, we will construct isomorphic generating trees for $\Inv\left(\underline{01}0\right)$ and $\Inv\left(\underline{01}1\right)$ by insertions on the right.
The following lemma is analogous to Lemma~\ref{lem:10-0_and_10-1}, with ascents playing the role of descents. Given $e\in\bI_{n}$, we say that $i$ is an {\it ascent} of $e$ if $e_{i}<e_{i+1}$, and let $\Asc(e)=\{i\in[n-1]:e_{i}<e_{i+1}\}$.

\begin{lem}\label{lem:01-0_and_01-1}
The active sites of $e\in\bI_n(\underline{01}1)$ are
$$\{0,1,\dots,n\}\setminus\{e_{i+1}:i\in\Asc(e)\}.$$
The active sites of $e\in\bI_n(\underline{01}0)$ are
$$\{0,1,\dots,n\}\setminus\{e_{i}:i\in\Asc(e)\}.$$
In particular, $e_n$ is an active site of $e\in\bI_n(\underline{01}0)$.
\end{lem}

\begin{proof} A value $h\in\{0,1,\dots,n\}$ is an active site of $e\in\bI_n(\underline{01}1)$ if and only if there does not exist $i<n$ such that $e_i<e_{i+1}=h$, and it is an active site of $e\in\bI_n(\underline{01}0)$ if and only if there does not exist $i<n$ such that $h=e_i<e_{i+1}$.

For the last statement, note that if $e_n$ was not an active site of $e\in\bI_{n}\left(\underline{01}0\right)$, there would exist $i\in\Asc(e)$ such that $e_{i}=e_{n}$, but then ${e_{i}e_{i+1}e_{n}}$ would be an occurrence of $\underline{01}0$, which is a contradiction.
\end{proof}

\begin{prop}\label{prop:01-0_and_01-1_1} The class $\Inv(\underline{01}1)$ has a generating tree described by the succession rule
\[
  \Omega_{\Inv\left(\underline{01}1\right)} = \begin{cases}
      (0,1), & \\
      (a,b)\rightsquigarrow & \hspace{-3mm}(a,b),(a-1,b+1),\ldots,(1,b+a-1), \\
        & \hspace{-3mm}(a+b,1),(a+b-1,2),\ldots,(a+1,b).
      \end{cases}
\]
\end{prop}

\begin{proof} We construct a generating tree by insertions on the right. To each $e\in\bI_{n}\left(\underline{01}1\right)$, we assign the label $(a,b) = \left(\left|A_{>}(e)\right|,\left|B_{\leq}(e)\right|\right)$, where
\begin{align}\label{eq:def_sets_sites_2}
\begin{split}
A_{>}(e) &= \left\{h:h\textnormal{ is an active site of }e\textnormal{ and }h>e_{n}\right\}, \\
B_{\leq}(e) &= \left\{h:h\textnormal{ is an active site of }e\textnormal{ and }h\leq e_{n}\right\},
\end{split}
\end{align}
with the convention $e_0=0$. The root, which is the empty inversion sequence, has label $(0,1)$.

Suppose now that $e\in\bI_{n}\left(\underline{01}1\right)$ has label $(a,b)$, and that we grow $e$ by inserting $h$ on the right, obtaining $eh\in\bI_{n+1}\left(\underline{01}1\right)$. The chosen active site $h$ must be either in $A_{>}(e)$ or in $B_{\leq}(e)$.

If $h$ is the $i$th smallest element in $B_{\leq}(e)$, then Lemma~\ref{lem:01-0_and_01-1} implies that $eh$ has label $(a+b+1-i,i)$,
considering the new active site $n+1$ of $eh$. This case is illustrated in Figure~\ref{fig:succession_4}(top). As $i$ ranges from $1$ to $b$, the resulting inversion sequences $eh\in\bI_{n+1}\left(\underline{01}1\right)$ have labels
\[
(a+b,1),(a+b-1,2),\ldots,(a+1,b).
\]

\begin{figure}[htb]
\begin{center}
\includegraphics[scale=0.52]{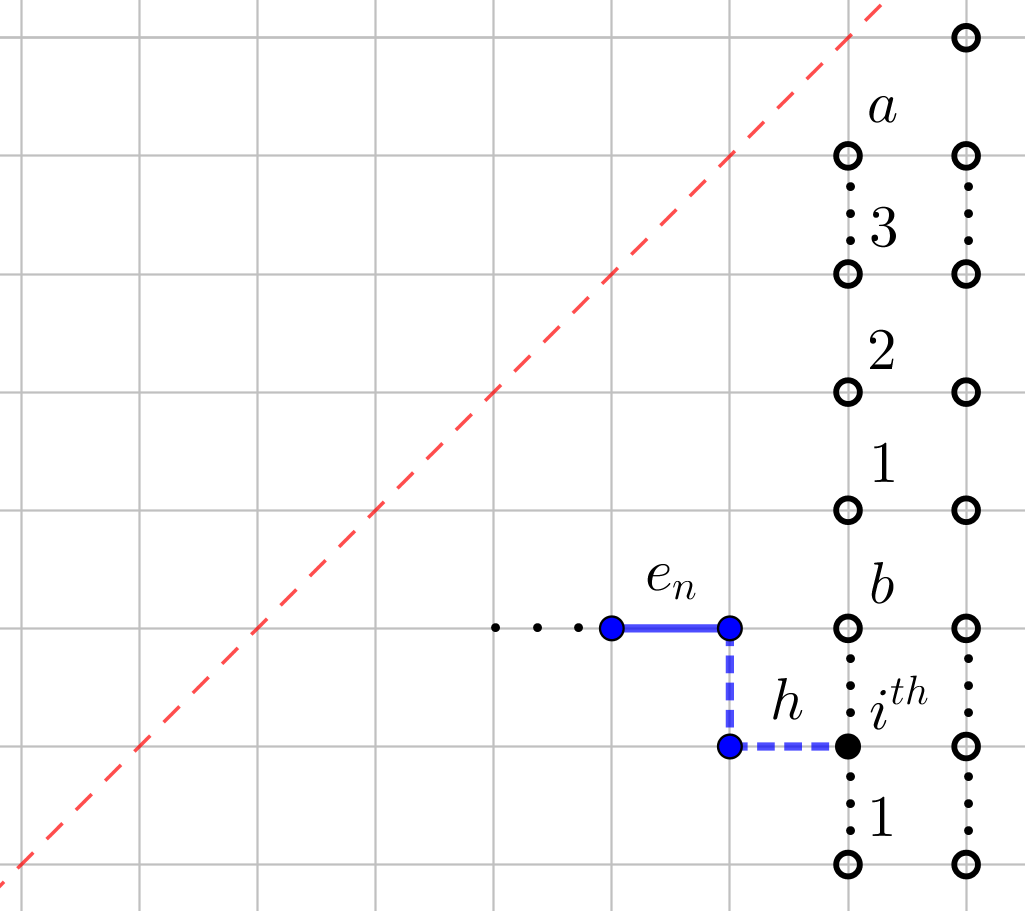}
\hspace*{1.5cm}
\includegraphics[scale=0.52]{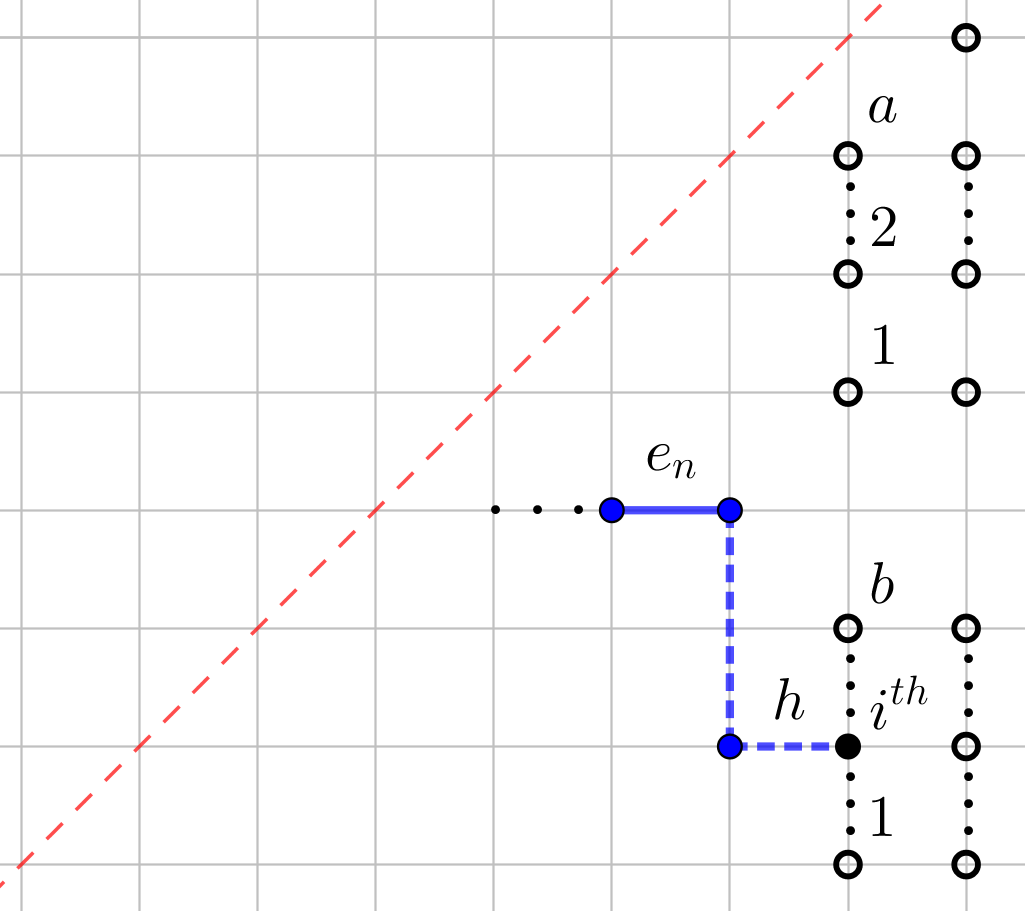}\bigskip

\includegraphics[scale=0.52]{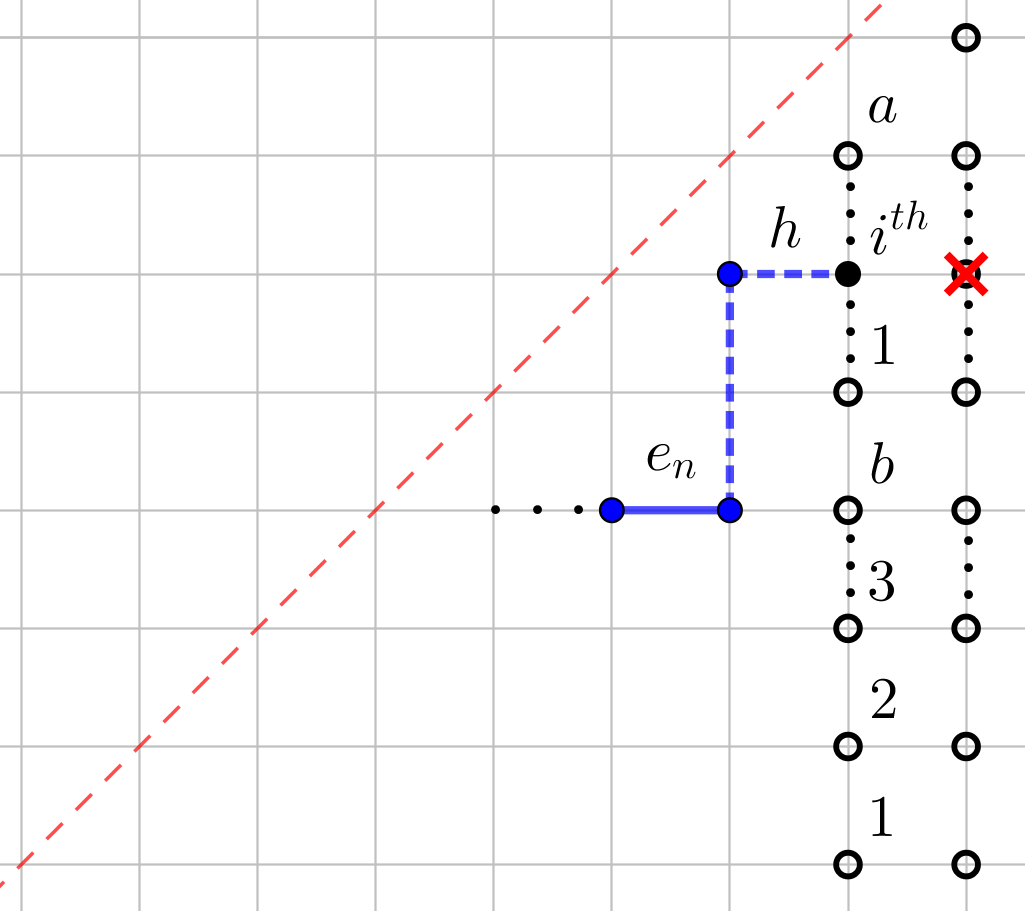}
\hspace*{1.5cm}
\includegraphics[scale=0.52]{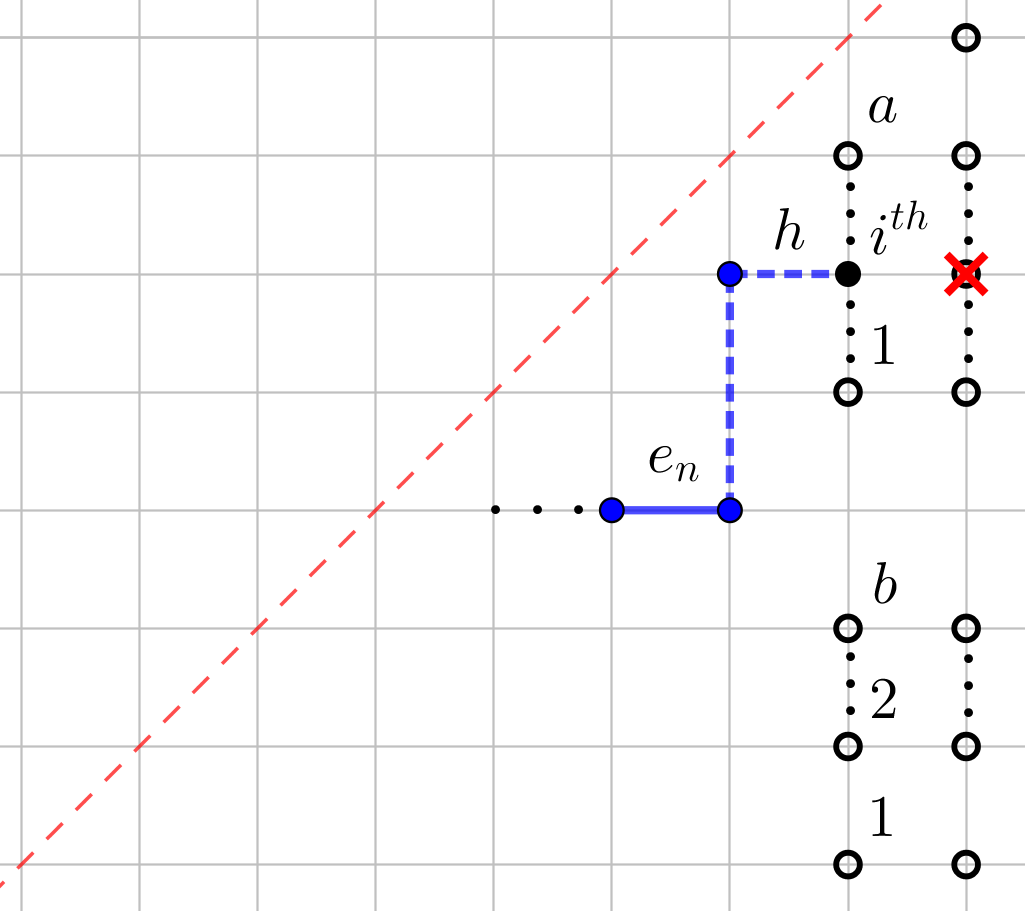}\par
\end{center}
\caption[Growth of $e\in\bI_{n}(\protect\underline{01}1)$]{Growth of $e\in\bI_{n}(\protect\underline{01}1)$ when the $i$th smallest active site in $B_{\leq}(e)$ (top) or $A_{>}(e)$ (bottom) is chosen. The site $e_{n}$ may be active (left) or not (right).}\label{fig:succession_4}
\end{figure}

If $h$ is the $i$th smallest element in $A_{>}(e)$, then $n$ is an ascent of $eh$, so Lemma~\ref{lem:01-0_and_01-1} implies that $h$ is not an active site of $eh$. Considering the new active site $n+1$, the inversion sequence $eh$ has label $(a+1-i,b-1+i)$, see Figure~\ref{fig:succession_4}(bottom).
As $i$ ranges from $1$ to $a$, the resulting inversion sequences in $\bI_{n+1}\left(\underline{01}1\right)$ have labels
\[
(a,b),(a-1,b+1),\ldots,(1,b+a-1).\qedhere
\]
\end{proof}

\begin{prop}\label{prop:01-0_and_01-1_2} The class $\Inv(\underline{01}0)$ has a generating tree described by the succession rule
\[
  \Omega_{\Inv\left(\underline{01}0\right)} = \begin{cases}
      (0,1), & \\
      (a,b)\rightsquigarrow & \hspace{-3mm}(a,b),(a-1,b+1),\ldots,(1,b+a-1), \\
        & \hspace{-3mm}(a+b,1),(a+b-1,2),\ldots,(a+1,b).
      \end{cases}
\]
\end{prop}

\begin{proof} We assign to each $e\in\bI_{n}\left(\underline{01}0\right)$ the label $(a,b) = \left(\left|A_{>}(e)\right|,\left|B_{\leq}(e)\right|\right)$, where $A_{>}(e)$ and $B_{\leq}(e)$ are as in Equation~\eqref{eq:def_sets_sites_2}.
As in the proof of Proposition~\ref{prop:01-0_and_01-1_1}, the root  has label $(0,1)$.
Given $e\in\bI_{n}\left(\underline{01}0\right)$ with label $(a,b)$, we grow $e$ by inserting an entry $h$ on the right so that $eh\in\bI_{n+1}\left(\underline{01}0\right)$.

If $h\in B_{\leq}(e)$, then all the active sites of $e$ are also active sites of $eh$ by Lemma~\ref{lem:01-0_and_01-1}, and so the resulting inversion sequences $eh$ for such $h$ have labels
\[
(a+b,1),(a+b-1,2),\ldots,(a+1,b).
\]
This case corresponds also to Figure~\ref{fig:succession_4}(top left), since $e_{n}$ is an active site of $e$ by Lemma~\ref{lem:01-0_and_01-1}.

The other possibility is that $h\in A_{>}(e)$. Suppose that $h$ is the $i$th smallest element in $A_{>}(e)$.
Then $n$ is an ascent of $eh$, and Lemma~\ref{lem:01-0_and_01-1} implies that $e_{n}$ is not an active site of $eh$, but $e_n$ was an active site of $e$. In this case, $eh$ has $i+(b-1)$ active sites $h'$ such that $h'\leq h$, namely the $i$ sites such that $e_{n}< h'\leq h$, and the $b-1$ sites such that $h'< e_{n}$. In addition, $eh$ has $(a-i)+1$ active sites $h'$ such that $h'>h$, once we include the site $n+1$. Hence, $eh$ has label $(a+1-i,b-1+i)$, see Figure~\ref{fig:succession_6}.
As $i$ ranges from $1$ to $b$, the resulting inversion sequences in $eh\in \bI_{n+1}\left(\underline{01}0\right)$ have labels
\[
(a,b),(a-1,b+1),\ldots,(1,b+a-1).\qedhere
\]
\end{proof}

\begin{figure}[htp]
\begin{center}
	\includegraphics[scale=0.52]{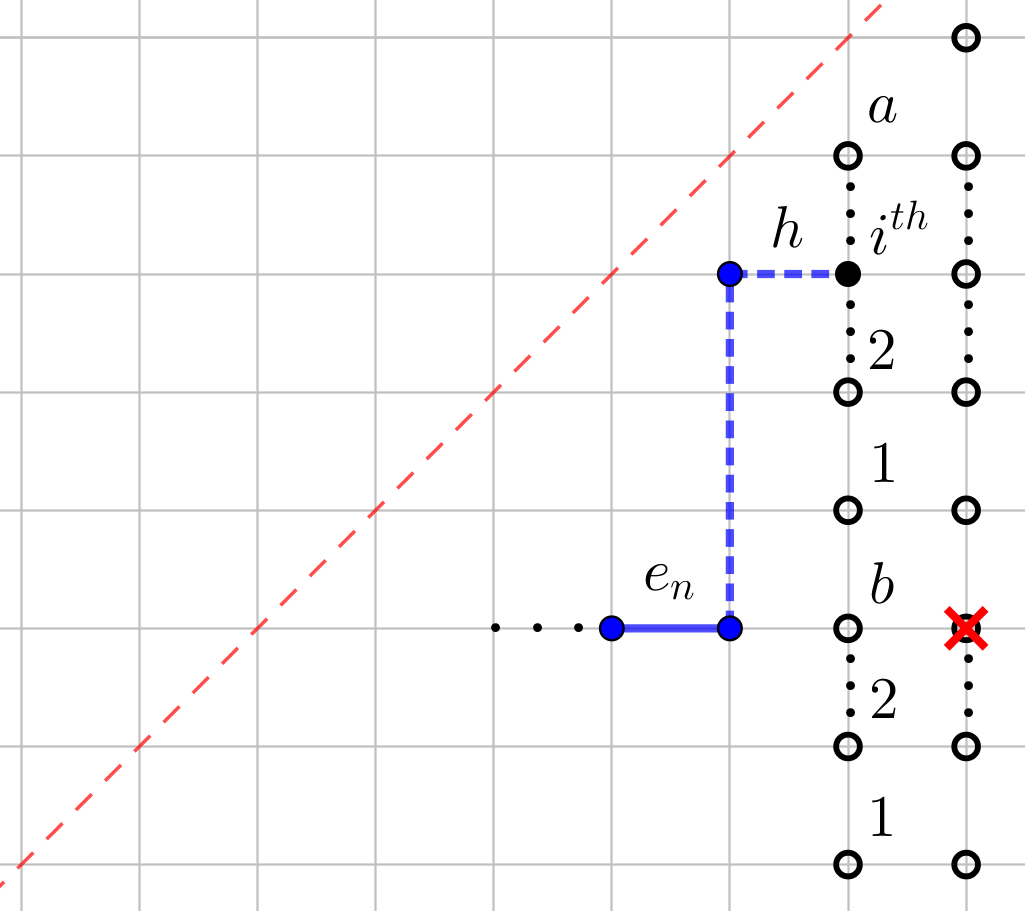}
\end{center}
\caption[Growth of $e\in\bI_{n}(\protect\underline{01}0)$]{Growth of $e\in\bI_{n}(\protect\underline{01}0)$ when the $i$th smallest active site in $A_{>}(e)$ is chosen.}\label{fig:succession_6}
\end{figure}

The generating trees for $\Inv\left(\underline{01}1\right)$ and $\Inv\left(\underline{01}0\right)$ described in Propositions~\ref{prop:01-0_and_01-1_1} and~\ref{prop:01-0_and_01-1_2} are isomorphic, and so the next result follows.

\begin{cor}\label{cor:01_0_and_01_1} The patterns $\underline{01}0$ and $\underline{01}1$ are Wilf equivalent.
\end{cor}

It is easy to check that these two patterns are not strongly Wilf equivalent: the are 52 inversion sequences of length 5 containing exactly one occurrence of $\underline{01}0$, but only 50 containing exactly one occurrence of $\underline{01}1$.

\begin{proof}[Proof of Theorem~\ref{EquivVinc}] By Propositions~\ref{prop:2_01_and_2_10} and~\ref{prop:1_01_and_1_10}, and Corollaries~\ref{cor:10_0_and_10_1}, and~\ref{cor:01_0_and_01_1}, we know that the equivalences (i)--(iv) hold. We provided computational evidence showing that $\underline{01}0 \stackrel{s}{\not\sim}\underline{01}1$, $\underline{10}0 \stackrel{s}{\not\sim}\underline{10}1$, $2\underline{01} \stackrel{s}{\not\sim}2\underline{10}$, and $1\underline{01} \stackrel{ss}{\not\sim}1\underline{10}$, and a brute force computation for small values of~$n$ shows that no two other hybrid vincular patterns are Wilf equivalent.
\end{proof}

\bibliographystyle{amsplain}
\addcontentsline{toc}{chapter}{References}
\bibliography{dartmouth_bib}

\end{document}